\documentclass{article}

\usepackage{amsthm}
\usepackage{amsmath}
\usepackage{amsfonts}
\usepackage{mathtools}
\usepackage{stix}
\usepackage{graphicx}
\usepackage{dsfont}
\usepackage{bbm}
\usepackage{tikz}
\usepackage{color}
\usepackage{stmaryrd}
\usepackage{appendix}
\usepackage{verbatim}
\usepackage{tikz-cd}
\usepackage[pdftex,color]{changebar}
\usepackage{thmtools}
\usepackage{thm-restate}

\usepackage{soul}
\newcommand{\rfname}[1]{{#1}}

\declaretheorem[name=Theorem,numberwithin=section]{theorem}
\declaretheorem[name=Lemma,numberwithin=section]{lemma}

\DeclareMathOperator*{\argmax}{arg\,max}
\DeclareMathOperator*{\argmin}{arg\,min}

\newtheorem{definition}[theorem]{Definition}
\newtheorem{example}[theorem]{Example}

\newtheorem{remark}[theorem]{Remark}

\begin{document}

\DeclareRobustCommand{\Triangle}
{%
\begin{tikzpicture}
    \draw (0,0) [fill=white] circle [radius=0.05] --(0.1,0.2) [fill=white] circle [radius=0.05] --  (0.2,0) [fill=white] circle [radius=0.05] -- (0,0) ;      
    \end{tikzpicture}
}

\DeclareRobustCommand{\ltrianglea}
{%
\begin{tikzpicture}
\fill[black] (0,0) circle [radius=0.05];
\fill[black] (0.1,0.2) circle [radius=0.05];
\draw [fill=white] (0.1,0.2)--(0.2,0) circle [radius=0.05] -- (0,0);
\end{tikzpicture}
}

\DeclareRobustCommand{\ltriangleb}
{%
\begin{tikzpicture}
\fill[black] (0.1,0.2) circle [radius=0.05] ;
\fill[black] (0.2,0) circle [radius=0.05] ;
\draw [fill=white] (0.2,0)--(0,0) circle [radius=0.05] -- (0.1,0.2);
\end{tikzpicture}
}

\DeclareRobustCommand{\ltrianglec}
{%
\begin{tikzpicture}
\fill[black] (0.2,0) circle [radius=0.05];
\fill[black] (0,0) circle [radius=0.05];
\draw [fill=white] (0,0)--(0.1,0.2) circle [radius=0.05] -- (0.2,0);
\end{tikzpicture}
}

\DeclareRobustCommand{\ltriangled}
{%
\begin{tikzpicture}
\fill[black] (0,0) circle [radius=0.05];
\fill[black] (0.1,0.2) circle [radius=0.05];
\draw [fill=white] (0.1,0.2)--(0.2,0) circle [radius=0.05] -- (0,0);
\end{tikzpicture}
}

\DeclareRobustCommand{\edgetriangle}
{%
W^*([\begin{tikzpicture}
   \draw (0,0) [fill=white] circle [radius=0.05] --(0,0.2) [fill=white] circle [radius=0.05]    ;  
   \end{tikzpicture},\begin{tikzpicture}
    \draw (0,0) [fill=white] circle [radius=0.05] --(0.1,0.2) [fill=white] circle [radius=0.05] --  (0.2,0) [fill=white] circle [radius=0.05] -- (0,0) ;      
    \end{tikzpicture}],[\rho, \tau])
}

\DeclareRobustCommand{\feasibleregionmedgetriangle}
{%
S^{(2)}([\begin{tikzpicture}
   \draw (0,0) [fill=white] circle [radius=0.05] --(0,0.2) [fill=white] circle [radius=0.05]    ;  
   \end{tikzpicture},\begin{tikzpicture}
    \draw (0,0) [fill=white] circle [radius=0.05] --(0.1,0.2) [fill=white] circle [radius=0.05] --  (0.2,0) [fill=white] circle [radius=0.05] -- (0,0) ;      
    \end{tikzpicture}],[\rho, \tau])
}

\DeclareRobustCommand{\feasibleregionedgetriangle}
{%
S_\mathbb{R}([\begin{tikzpicture}
   \draw (0,0) [fill=white] circle [radius=0.05] --(0,0.2) [fill=white] circle [radius=0.05]    ;  
   \end{tikzpicture},\begin{tikzpicture}
    \draw (0,0) [fill=white] circle [radius=0.05] --(0.1,0.2) [fill=white] circle [radius=0.05] --  (0.2,0) [fill=white] circle [radius=0.05] -- (0,0) ;      
    \end{tikzpicture}],[\rho, \tau])
}

\DeclareRobustCommand{\Edge}
{%
\begin{tikzpicture}
   \draw (0,0) [fill=white] circle [radius=0.05] --(0,0.2) [fill=white] circle [radius=0.05]    ;  
   \end{tikzpicture}
}

\DeclareRobustCommand{\ErdosRenyiedgetriangle}
{%
W^*([\begin{tikzpicture}
   \draw (0,0) [fill=white] circle [radius=0.05] --(0,0.2) [fill=white] circle [radius=0.05]    ;  
   \end{tikzpicture},\begin{tikzpicture}
    \draw (0,0) [fill=white] circle [radius=0.05] --(0.1,0.2) [fill=white] circle [radius=0.05] --  (0.2,0) [fill=white] circle [radius=0.05] -- (0,0) ;      
    \end{tikzpicture}],[\rho, \tau])
}

\numberwithin{equation}{section}

\newcommand{\abs}[1]{\lvert#1\rvert}

\newcommand{\blankbox}[2]{%
  \parbox{\columnwidth}{\centering
    \setlength{\fboxsep}{0pt}%
    \fbox{\raisebox{0pt}[#2]{\hspace{#1}}}%
  }%
}

\newtheorem{problem}{Problem}
\newtheorem{Remark}{Remark}
\newtheorem{algorithm}{Algorithm}

\newcommand{\sfmat}{A}
\newcommand{\extremestats}[1]{\hbox{Extr}(#1)}
\newcommand{\cristats}[2]{\hbox{Cri}^{(#1)}(#2)}
\newcommand{\cristatspi}[3]{\hbox{Cri}^{(#1)}(#2,#3)}

\title{Constrained Multi-Relational Hyper-Graphons with Maximum Entropy}

\newcommand{\jan}[1]{\textcolor{blue}{{#1}}}
\newcommand{\juan}[1]{\textcolor{red}{{#1}}}
\newcommand{\phdExt}[1]{\textcolor{yellow}{\textit{[}}{#1}\textcolor{yellow}{\textit{]}}}
\newcommand{\janfoot}[1]{\textcolor{blue}{{\footnote{\jan{#1}}}}}
\newcommand{\juanfoot}[1]{\textcolor{red}{{\footnote{\juan{#1}}}}}
\newcommand{\janmargin}[1]{\marginpar{\jan{{#1}}}}
\newcommand{\juanmargin}[1]{\marginpar{\juan{{#1}}}}

\newcommand{\graphonspaceu}{\widetilde{\mathcal{W}}^{(r,d)} }
\newcommand{\graphonspaceuu}[2]{\widetilde{\mathcal{W}}^{(#1,#2)} }
\newcommand{\graphonspacel}{\mathcal{W}^{(r,d)} } 
\newcommand{\graphonspacell}[2]{\mathcal{W}^{(#1,#2)} } 
\newcommand{\realgraphonspaceu}{\widetilde{\mathcal{W}}_{\mathbb{R}}^{(r,d)} }
\newcommand{\realgraphonspacel}{\mathcal{W}_{\mathbb{R}}^{(r,d)} } 
\newcommand{\realstepfunctionspacel}[1]{\mathcal{W}^{({#1,r,d})}_{\mathbb{R}}}
\newcommand{\realstepfunctionspaceu}[1]{{\widebridgeabove{\mathcal{W}}}^{({#1,r,d})}_{\mathbb{R}}}
\newcommand{\realstepfunctionspaceupi}[2]{{\widebridgeabove{\mathcal{W}}}^{({#1},r,d)}_{\mathbb{R}}({#2})}
\newcommand{\realstepfunctionspacelpi}[2]{{\mathcal{W}}^{({#1},r,d)}_{\mathbb{R}}(#2)}

\newcommand{\realfeasibleregionl}{S^{(r,d)}_{\mathbb{R}}(\mathcal{F},u)}
\newcommand{\realfeasibleregionmu}[1]{{{\widebridgeabove{S}}^{(r,d)}_{\mathbb{R}}^{(#1)}(\mathcal{F},u)}}
\newcommand{\realfeasibleregionml}[1]{{{S}_{\mathbb{R}}^{(#1,r,d)}(\mathcal{F},u)}}
\newcommand{\realfeasibleregionmlpi}[2]{{S_{\mathbb{R}}^{(#1,r,d)}(\mathcal{F},u,#2)}}

\newcommand{\graphonPermutationSpace}{\Sigma}
\newcommand{\graphonPermutationSpacem}[1]{\Sigma^{({#1})}}

\newcommand{\eqclass}[1]{[{#1}]_{\sim}}
\newcommand{\convexcombm}[1]{P^{({#1})}}
\newcommand{\convexcombmF}[1]{P^{({#1})}(\mathcal{F})}
\newcommand{\convexcombmFu}[1]{P^{({#1})}(\mathcal{F},u)}

\newcommand{\lebesgue}{\nu}  
\newcommand{\partialF}{\partial F}

\newcommand{\criticalpoints}[1]{Cr^{(#1)}(\mathcal{F},f_s)}
\newcommand{\criticalpointspi}[2]{Cr^{(#1)}(\mathcal{F},#2,f_s)}
\newcommand{\criticalpointspiI}[2]{Cr^{(#1)}(\mathcal{F},#2,I)}
\newcommand{\minimalpoints}[1]{Cq^{(#1)}(\mathcal{F},f_s)}

\newcommand{\stepfunctionspacel}[1]{{\mathcal{W}}^{({#1,r,d})}}
\newcommand{\stepfunctionspacelpi}[2]{{\mathcal{W}}^{({#1,r,d})}(#2)}

\newcommand{\optimalsolutionf}[1]{W^{*}(\mathcal{F},u,#1)}
\newcommand{\optimalsolutionmf}[2]{W^{*(#1)}(\mathcal{F},u,#2)}
\newcommand{\optimalsolutionmpif}[2]{W^{*(#1)}(\mathcal{F},u,#2)}

\newcommand{\marginalpolytopem}[1]{T^{(#1)}(\mathcal{F})}
\newcommand{\totalmarginalpolytope}{T(\mathcal{F})}

\newcommand{\reg}[1]{\hbox{Reg}(#1)}
\newcommand{\regstats}[1]{\hbox{Reg}(#1)}
\newcommand{\regstatsm}[2]{\hbox{Reg}^{(#1)}(#2)}
\newcommand{\regstatsmz}[2]{\hbox{Reg}_0^{(#1)}(#2)}
\newcommand{\regstatspi}[3]{\hbox{Reg}^{(#1)}(#2,#3)}

\newcommand{\realfeasibleregionmq}[3]{{S}^{({#1,r,d})}_{\mathbb{R}}({#2},{#3})}
\newcommand{\crealfeasibleregionmq}[3]{{\overline{S}}^{({#1,r,d})}_{\mathbb{R}}({#2},{#3})}
\newcommand{\realfeasibleregionmllpi}[4]{{{S}}^{({#1,r,d})}_{\mathbb{R}}({#2},{#3},{#4})}

\newcommand{\optimalsolution}{W^{*(r,d)}(\mathcal{F},u)}
\newcommand{\realoptimalsolution}[1]{W_{\mathbb{R}}^{*(r,d)}(\mathcal{F},u, #1)}
\newcommand{\realoptimalsolutionm}[2]{W_{\mathbb{R}}^{*(#1,r,d)}(\mathcal{F},u, #2)}

\newcommand{\feasibleregion}{\widetilde{S}^{(r,d)}(\mathcal{F},u)}
\newcommand{\realfeasibleregionu}{\widetilde{S}_{\mathbb{R}}(\mathcal{F},u)}
\newcommand{\feasibleregionml}[1]{{{S}^{(#1,r,d)}(\mathcal{F},u)}}
\newcommand{\randfeasibleregionml}[1]{{{S}_{(0,1)}^{(#1)}(\mathcal{F},u)}}

\author{Juan Alvarado\footnote{jaoxxx3@gmail.com} \thanks{ KU Leuven, Department of Computer Science}  \and  Jan Ramon\footnote{jan.ramon@inria.fr} \thanks{INRIA Lille}  \and Yuyi Wang \footnote{yuwang@ethz.ch} \thanks{ETH Zurich, Switzerland \& CRRC Zhuzhou Institute.}  }



\date{}


\maketitle

\begin{abstract}
  
This work has two contributions. The first one is extending the Large Deviation Principle for uniform hyper-graphons from Lubetzky and Zhao~\cite{lubetzky2015replica} to the multi-relational setting where each hyper-graphon can have different arities. This extension enables the formulation of the most typical possible world in Relational Probabilistic Logic with symmetric relational symbols in terms of entropy maximization subjected to constraints of quantum sub-hypergraph densities.

The second contribution is to prove the most typical constrained multi-relational hyper-graphons (the most typical possible worlds) are computable by proving the solutions of the maximum entropy subjected by quantum sub-hypergraph densities in the space of multi-relational hyper-graphons are step functions except for in a zero measure set of combinations of quantum hyper-graphs densities with multiple relations. 
 This result proves in a very general context the conjecture formulated by Radin et al.\ \cite{radin2014asymptotics} that states the constrained graphons with maximum entropy are step functions.

\end{abstract}

\noindent {\small  keywords: Graphon Theory and Principle of Maximum Entropy and Large Random Graph and Exponential Random Graph Model and Constrained optimization and Differential Geometry and Statistical Relational Learning }

\section{Introduction}
\label{sec:intro}

In the space of graphons, the problem of finding the most typical random graphs satisfying subgraph density constraints is formulated in terms of entropy maximization which is described as follows \cite{radin2013phase}. Let $\mathcal{F}$ be an ordered set of simple graphs and $u \in (0,1)^{|\mathcal{F}|} $  be a vector of prescribed subgraph densities.
\begin{equation*}
    W^{*}(\mathcal{F},u) = \argmax_{W \in S(\mathcal{F},u)} -I(W)
\end{equation*}
where 
\begin{equation*}
    S(\mathcal{F},u) = \{ W \in \mathcal{W} \, | \, t(\mathcal{F}_i,W ) = u_i \mbox{ for all } i = 1, \cdots , |\mathcal{F}|| \} 
\end{equation*}
and $-I$ is the Shannon entropy on graphons and $\mathcal{W}$ the space of labeled graphons\footnote{This is formulation originally was established in the space of unlabeled graphons, but for the sake of simplicity, we formulate in terms of labeled graphons.}.  Here, we are interested in extending the formulation to the most typical multi-relational hyper-graphs regarding entropy maximization.


Lubetzky and Zhao \cite{lubetzky2015replica} extend the maximum entropy principle of graphons to uniform hyper-graphons and prove the counting lemma holds for linear hyper-graphs. Here, we further extend this principle to multi-relational hyper-graphons (Section \ref{sec:hypergraphons}) following the same ideas of our previous work \cite{alvarado2022limits} and the contraction principle of Large Deviation Principles \cite{dembo2009large}. 
There are two motivations for establishing this extension.
\begin{itemize}
    \item To establish a micro-canonical semantics of possible worlds of probabilistic logics whose signatures are only relational symbols describing symmetric relations, the symbols $\wedge$ and $\implies$ and universal quantifiers. In this probabilistic logic, each axiom is associated with a probability number.

    \item To prove the problem of finding the most typical worlds (the micro-canonical distribution) satisfying the probabilistic axioms can be computable by proving the solutions of the maximum entropy problem in the space of multi-relational uniform hyper-graphons are step functions. This result proves a general version of the conjecture raised by Radin et al.\ \cite{radin2014asymptotics}, which states $W^*(\mathcal{F},u)$ is a step function.

\end{itemize}

Since the formulation of this conjecture, it has been proved for several special cases. The conjecture was rigorously confirmed by Kenyon et al.\ \cite{kenyon2017multipodal} when $\mathcal{F}$ contains the edge and a $k$-stars graph. Moreover Kenyon et al.\ in \cite{kenyon2016bipodal} confirm the conjecture is true  when $\mathcal{F} = [\begin{tikzpicture}
   \draw (0,0) [fill=white] circle [radius=0.05] --(0,0.2) [fill=white] circle [radius=0.05]    ;  
   \end{tikzpicture},F]$ and $u=[\rho, \tau]$ when  $\tau$ is slightly higher than $\rho^{|F|}$,  where $F$ is a simple graph and $|F|$ is the number of edges of $F$.  Moreover, Kenyon et al.\ in \cite{kenyon2017phases} show that the number of steps that $W^{*}(\mathcal{F},u)$ has may not be bounded when the vector $u$ gets arbitrarily close to an extremal sufficient statistics vector.
in the case of edge and triangle constraints. 

Roughly speaking,  our approach to proving the conjecture is based on the observation any step function can be seen as a function and a finitely dimensional vector of parameters that describe the step function. We study the solutions $W^{*}(\mathcal{F},u)$ restricted to the space of step functions of size $m$, these solutions are denoted by $W^{*(m)}(\mathcal{F},u)$.   Using elementary results of differential geometry and constrained optimization, we prove when the solutions $W^{*(m)}(\mathcal{F},u)$ are embedded in a higher dimensional vector space whose vectors parameterize step functions of size $m+1$ then $W^{*(m)}(\mathcal{F},u)$ are also solutions in the extended space. In other words, we prove $W^{*(m)}(\mathcal{F},u) = W^{*(m+1)}(\mathcal{F},u)$.

\subsection{Outline}

Thus, this paper is organized as follows. Section \ref{sec:mathematicalbacground} provides results from \rfname{differential geometry} and constrained optimization to support the proofs. Section \ref{sec:hypergraphons} develops the notion of multi-relational hyper-graphons. 
Section \ref{sec:stepfunctionspace} studies the properties of step functions in the space of multi-relational hyper-graphons providing formulas to compute partial derivatives of subgraph densities when the step functions are finitely parameterized and defining the split map on step functions which embeds any stepfuncion of lower dimensional space into a higher one. Section  \ref{sec:uniformmulti-relationalhypergraphons} extends key results of graphon theory as Szemeredi Regularity Lemma, Compactness, Counting Lemma, and Large Deviation Principle to multi-relational uniform hyper-graphons where each relation has the same arity. Section \ref{sec:nonuniformhypergraphons} extends the results developed in Section \ref{sec:uniformmulti-relationalhypergraphons}, we obtain the large deviation principle for Erd\H os-Renyi version of multi-relational random hyper-graphs and prove the most typical random hyper-graphs constrained by quantum subgraph densities are solution of a maximum entropy problem in the space of multi-relational hyper-graphs.  Section \ref{sec:possibleworlds} provides an example of modeling possible worlds in a probabilistic logic using the most typical multi-relational random hyper-graphs.  Section \ref{sec:computability} proves that the most typical multi-relational random hyper-graphs are step functions.
 Finally, Section \ref{sec:conclusions}  provides concluding remarks and some open problems.  The Appendix contains the proofs of theorems and lemmas.

\subsection{Notations}
\label{sec:prelim}

Let us start with some common notations. 
We denote by $\mathbbm{1}_X$ the indicator function, i.e., for any $x$, if $x\in X$ then $\mathbbm{1}_X(x)=1$ else $\mathbbm{1}_X(x)=0$.
For any positive integer $n$, we denote by $[n]$ the set of all positive integers smaller than or equal to $n$. Let $S$ be a set. Then any sequence in $S$ is denoted by $(a_i)$ where $a_i \in S$.  

We denote by $\nu$ the  Lebesgue measure on reals. 
We denote by $\graphonPermutationSpace$  all bijective measure-preserving maps and by  $\Sigma_m$ the symmetric group on $[m]$. We use the standard notation of  $A^\circ$, $\overline{A}$ and $\partial A$ for the topological interior, closure, and boundary of the set $A$ and by $|A|$ the cardinality of a finite set $A$.  
The set of finite, simple, and non-isomorphic graphs is denoted by  $ \mathcal{G}$. A multi-relational graph $F$ is defined by $(V,E_1, \cdots, E_r)$ where $V$ is the set of vertices and $(E_1, \cdots, E_r)$ are the set of relations. The number of vertices is denoted $|F|$.

Let $F: \mathbb{R}^n \to \mathbb{R}^m $ be a smooth function then $J_x(F)$ denotes the Jacobian matrix of $F$ at $x$. Let $f: \mathbb{R}^n \to \mathbb{R}$ be a smooth function, $H_{x} f$ denotes  the Hessian matrix of $f$ at $x$.

\begin{table}[htb]
\begin{tabular}{|ll|}
\hline
Symbol & Meaning \\\hline
$\realgraphonspacel$ & All labeled graphon with real-valued   \\ 
$\realgraphonspaceu$ & All unlabeled graphons the quotient of $\realgraphonspacel$ by elements of $\Sigma$  \\ 
$\graphonspacel$ & All labeled graphon with $[0,1]$-valued   \\ 
$\graphonspaceu$ & The quotient of $\graphonspacel$ by elements of $\Sigma$    \\ 
$(\sfmat,\pi)$& A step function with partition   $\pi$(Sec.\ref{sec:stepfunctionspace})  \\ 
$\realstepfunctionspacel{m}$ & All $m$-step functions with real-valued (Sec. \ref{sec:stepfunctionspace})        \\ 
$\realstepfunctionspacelpi{m}{\pi}$ & All $m$-step functions with real-valued with a fixed partition vector $\pi$\\ 
$\realfeasibleregionl$& All unlabeled graphons with real-valued satisfying constraints  (\ref{eq:feasibleregion}) \\
$\realfeasibleregionml{m}$& All $m$-step functions with real-valued satisfying constraints.\\
$\realfeasibleregionmlpi{m}{\pi}$& All $m$-step functions with real-valued satisfying constraints and $\pi$ fixed.\\
$t(F,W)$  & Subgraph density of $F$ from the graphon $W$  (\ref{eq:subgraphdensity})\\ $t_{x_1 \cdots x_k }(F,W)$ & Partial subgraph density  (\ref{eq:condsubgraphdensity})\\ 
\hline 
\end{tabular}
\caption{\label{tab:notations} Overview of important notations in this paper}
\end{table}

\newcommand{\varQuot}[1]{{{\color{brown}{[varQuot] {#1}}}}}
\newcommand{\varQuotDelStart}{\cbcolor{brown}\begin{changebar}}
\newcommand{\varQuotDelEnd}{\end{changebar}}

\section{Mathematical Background}
\label{sec:mathematicalbacground}

This section reviews results from Differential Geometry and Optimization on constrained sets.

\subsection{Riemannian Geometry}

Riemannian Geometry endows a smooth manifold $M$ with a geometric structure by defining a positive definite bilinear form $g_p: T_p M \otimes T_p M \to \mathbb{R}$ at each point $p \in M$, such that the mapping $p \to g_p(X; Y)$ is smooth. This structure allows us to measure distances and angles on the manifold.

By utilizing the Riemannian metric $g_p(\cdot, \cdot)$, we can define the length of a curve $\gamma: [a, b] \to M$ in the following manner: Let $g_{p;ij}$ denote the components of the Riemannian metric, and let $y^{(i)}$ represent the coordinates of the curve $\gamma$. Then, the path length $L(\gamma)$ is given by the integral

\begin{equation*}
L(\gamma) = \int_a^b \sqrt{\sum_{ij} g_{\gamma(s);ij} \frac{d\gamma^{(i)}}{ds} \frac{d\gamma^{(j)}}{ds}}  ds.
\end{equation*}

This formulation quantifies the length of a curve on the manifold $M$ based on the underlying Riemannian metric. It enables us to compute distances and study the geometry of $M$ through the lens of Riemannian Geometry.

\subsubsection{Exponential Map on Riemannian Manifolds.}
 
In the context of Riemannian manifolds, geodesic curves serve as paths with minimal distance between points. The notion of a geodesic curve is formalized as follows:

\begin{definition}[Definition $1.4.2$ in \cite{jost2008riemannian}]
Let $M$ be a Riemannian manifold. A curve $\gamma: [0, a] \to M$ is said to be a geodesic if it satisfies the second-order differential equation:
\begin{equation}
\label{geodesicequation}
\frac{d^2\gamma^{(i)}}{ds^2} + \Gamma^{i}{j,k} \frac{d\gamma^{(j)}}{ds} \frac{d\gamma^{(k)}}{ds} = 0, \quad i, j, k = 1, \ldots, m,
\end{equation}
where $\gamma^{(i)}$ denotes the $i$-th component of $\gamma$ with respect to a local coordinate system, and $\Gamma^{i}{j,k}$ are the Christoffel symbols associated with the Levi-Civita connection.
\end{definition}

A fundamental tool in studying Riemannian manifolds is the \emph{Exponential Map}, denoted as $Exp_p: U \to M$, which provides a correspondence between tangent vectors and points on the manifold. Specifically, for each point $p \in M$, there exists an open neighborhood $U$ of $0 \in T_p M$ such that $Exp_p$ is defined as follows:

\begin{equation}
\label{eq:exponentialmap}
Exp_p(v) = \gamma_{p,v}(1),
\end{equation}
where $\gamma_{p,v}(1)$ represents the unique geodesic curve that satisfies the initial conditions $\gamma_{p,v}(0) = p$ and $\gamma'_{p,v}(0) = v$.

In other words, the Exponential Map takes a tangent vector $v \in T_p M$ and "exponentiates" it to a point $q = Exp_p(v) \in M$ along the geodesic that starts at $p$ with the velocity $v$.

A crucial property, as established by Proposition 20.8 in \cite{lee2003smooth}, is that $Exp_p$ is a local diffeomorphism when restricted to an open neighborhood of $0 \in T_p M$. This property ensures the existence of a well-defined inverse map, which is essential for various applications in Riemannian Geometry.

\subsection{Theory of Constrained Optimization}
\label{sec:optimization}

Here, we assume  $h: \mathbb{R}^n \to \mathbb{R}^k$ is differentiable. Then, we define the constrained set
\begin{equation}
\label{eq:constrainedset}
M_h = \{x \in \mathbb{R}^n \, | \, h(x) = 0\}.
\end{equation}
We review the Karush-Kuhn-Tucker (KKT) conditions for which any $x^* \in M_h$ is a local solution of $\min_{x \in M_h} f(x)$. The contents of this section are from 
\cite{jorge2006numerical}.

\begin{definition}[Lagrangian of a function]
Let $f: \mathbb{R}^n \to \mathbb{R}$ be a differentiable function and let $\beta \in \mathbb{R}^k$. Then the Lagrangian function of  $f$ on $x \in M_h$ the constraints $h=0$,
\begin{equation*}
    \mathcal{L}(f)(x,\beta )  = f(x) - \beta^\top h(x).
\end{equation*}
\end{definition}

\begin{theorem}[First-Order Necessary Conditions]
 If $x^* \in M_h$ is a local minimum of $\min_{x \in M_h} f(x)$ then there is a Lagrange multiplier vector $\beta^* \in \mathbb{R}^k$ such that  
\begin{equation*}
 \nabla_{x^*} \mathcal{L}(f)(x, \beta^*) = 0,    
\end{equation*}
 thus 
 \begin{equation*}
\nabla_{x^*} f = \beta^{*\top} \nabla_{x^*} h.
 \end{equation*}
 Hence $x^* \in M_h$ must be a critical point of $f:M_h \to \mathbb{R}$.
\end{theorem}
The sufficient condition for the optimality of $x^*$ comes from second-order conditions; hence, we need feasible directions.
\begin{definition}[The set of linearized feasible directions]
The set of linearized feasible directions on $M_gh$ at $x \in M_h $ corresponds to vectors orthogonal to the gradient of the constraints of $M_h$, i.e.
\begin{equation*}
    C_{x} M_h = \{ v | v^\top J_x h = 0\}.
\end{equation*}  
It is clear when the columns  $J_x h$ of the constraints of $M_h$ are linearly independent, the tangent space $T_x M_h$ is defined. 

\end{definition}
First, we have a necessary second-order condition for the optimality. Let $H_x^g f$ be the geodesic hessian of $f$ on $M_h$ defined by $H_x^g f = H_x \mathcal{L}(f)$.
\begin{theorem}[Second-order necessary conditions]
\label{thm:necessarysecondordercondition}
Let $x^*$ be a critical point of $f:M_h \to \mathbb{R}$. Then if $x^*$ is a local minimum  $f$ then $v^\top \left( H_x^g f \right) v \geq 0$ for all $v \in C_{x^*} M_h$.
\end{theorem}
\noindent and a sufficient condition,
\begin{theorem}[Second-order sufficient conditions]
\label{thm:sufficientsecondordercondition}
Let $x^*$ be a critical point of $f:M_h \to \mathbb{R}$. Then if $v^\top H_{x^*}^g f v > 0$  for all  $v \in C_{x^*} M_h \setminus \{0\}$, then $x^*$ is a strict local minimum of $f$ on $M_h$.
\end{theorem}

\subsection{Preservation of critical Points by Diffeomorphisms}

Here, we review that diffeomorphisms preserve critical points, the positive/negative definiteness, and the full rankness of Hessian matrices at any critical points of $f$ on $M$. 

Let $M$ be a smooth manifold with dimension $m$.  Since $M$ and $\mathbb{R}^m$ are locally diffeomorphic. Let $f:M \to \mathbb{R}$ be a smooth function, let $u^* \in \mathbb{R}^m$ be a critical point of $f$ and let $\phi:\mathbb{R}^m \to M$ be a local diffeomorphism at $u$ and $x^* = \phi(u^*)$.  By taking the first and second derivatives of $f \circ \phi$, we have 

\begin{eqnarray*}
\nabla_u f(x(u))  = \nabla_x f(x(u))J_u(x(u)), 
\end{eqnarray*}
and 
\begin{equation*}
H_u f(x(u)) = J_u(x(u))^\top H_x f(x(u))J_u(x(u)) + \nabla_x f(x(u))H_u x(u),
\end{equation*}
where $J_u(x(u))$ is a Jacobian matrix of $\phi$. Since $\phi:\mathbb{R}^m \to M$ is a diffeomorphism, $J_u(x(u))$ is invertible and if $u^*$ is a critical point then $\nabla_{u^*} f(x(u))=0$. Thus $x^*=\phi(u^*)$ is a critical point since $\nabla_x f(x(u^*))=0$. Moreover if $u^*$ is critical point  then  $H_{u^*} f \circ \phi$ is positive/negative (semi)definite iff  $H_{x(u^*)} f $ is positive/negative (semi)definite. Hence, we have the following lemma, 

\begin{lemma} 
\label{lem:diffeomorphismpreservation}
Let  $M$ be a smooth manifold with dimension $m$  and let $f:M \to \mathbb{R}$ be a smooth function. Let $\phi: \mathbb{R}^m \to M$ be a diffeomorphism and let $H_x f$ be the Hessian matrix of $f$ at $x \in M$. Let  $u^*$ be a critical point of $f$ and $u^*=\phi(x^*)$ then  $H_{u^*} f \circ \phi$ is positive/negative (semi)definite iff $H_{x^*} f$ is positive/negative (semi)definite.
\end{lemma}

\section{Hypergraphons}
\label{sec:hypergraphons}


A multi-relational hyper-graph $G=(V, E_1, \cdots, E_r)$ is a tuple of sets where $V$ is the set of vertices, and $(V,E_i)$ is a uniform hyper-graph, and $E_i$ is a relation on $V$ with arity $d_i$. A hyper-graph is uniform when all edges have the same arity. A uniform hyper-graph is linear when each pair of vertices has at most one hyperedge connecting them. 

Let $r$ be a positive integer and $d$ be a vector of positive integers of $r$ dimensions. Then, a $(r,d)$-hyper-graph is the tuple  $(V, E_1, \cdots, E_r)$ where $V$ is the set of vertices and each  $E_i$ is a set of uniform hyper-edges with arity $d_i$.  $(r,d)$-hyper-graphs are intended to be possible worlds (semantic models) of first-order languages with a signature of $r$ symmetric relational symbols where each relational symbol $E_i$ has an arity of $d_i$.

A multi-relational hypergraph with $r$ relations is the tuple $(V, E_1, \cdots, E_r)$ where $V$ is the set of vertices and each  $E_i$ is a set of $d_i$-hyper-edges. We denote 
Let $W: \prod_{k=1}^r( [0,1]^{d_k} \mapsto \mathbb{R})$ be the cartesian product of $r$ functions $[0,1]^{d_k} \mapsto \mathbb{R}$. Hence each coordinate is $W_k:[0,1]^{d_k} \to \mathbb{R}$.

The $(r,d)$-hyper-graphs intend to model the limit of a growing sequence of hyper-graphs $(r,d)$. The space of $(r,d)$-hyper-graphons is denoted by $\graphonspacel$ and defined by 
\begin{eqnarray*}
\graphonspacel = \{ W: \prod_{k=1}^r( [0,1]^{d_k} \mapsto [0,1]) \mid W_k(x_1, \cdots, x_{d_k}) \\ =W_k(x_{\psi(1)}, \cdots, x_{\psi(d_k)}) \mbox{ for } \psi \in \Sigma_d, k \in [r] \}.
\end{eqnarray*}
 The elements of $\graphonspacell{r}{d}$ can be considered tuples of non-parametric stochastic block models of random hypergraphs.  

Now, it is convenient to define the space of $d$-hypergraphons with $r$ relations on reals is 
\begin{eqnarray*}
\realgraphonspacel = \{ W: \prod_{k=1}^r( [0,1]^{d_k} \mapsto \mathbb{R}) \mid W_k(x_1, \cdots, x_{d_k}) \\ =W_k(x_{\psi(1)}, \cdots, x_{\psi(d_k)}) \mbox{ for } \psi \in \Sigma_d , k \in [r]\} 
\end{eqnarray*}
In general, if    $A \subseteq \mathbb{R}$ then
\begin{eqnarray*}
    \graphonspacel_{A}= \{  W: \prod_{k=1}^r( [0,1]^{d_k} \mapsto A) \mid W_k(x_1, \cdots, x_{d_k}) \\ =W_k(x_{\psi(1)}, \cdots, x_{\psi(d_k)}) \mbox{ for } \psi \in \Sigma_d , k \in [r]\}     
\end{eqnarray*}
We call elements of $\graphonspacel_{(0,1)}$  and  $\graphonspacel_{\mathbb{R}}$   respectively purely random graphons and  real-valued graphons and when $A$ is omitted then $A=[0,1]$.

Let $\sim$ be the equivalence relation on $\realgraphonspacel $ defined by $W \sim V$ iff it exists  $\sigma \in \Sigma$ such that for all $k \in [r]$ and $(x_1, \cdots,x_{d_k}) \in \mathbb{R}^{d_k}$  we have   $V_k^\sigma(x_1, \cdots,x_{d_k}) = V(\sigma(x_1), \cdots, \sigma(x_d{d_k})$ which is abbreviated by $W=V^\sigma$.
 The quotient spaces of $\graphonspacel$ and $\realgraphonspacel$ by $\sim$ are the spaces unlabeled d-hypergraphons with $r$ relations i.e. $\graphonspaceu=\graphonspacel/\sim$ and $\realgraphonspaceu=\realgraphonspacel/\sim$.

We endow  $\realgraphonspacel$ a topology by the cut-norm 
\begin{equation*}
    \| W \|_\Box = \sum_{k=1}^r \sup_{S_1, \cdots S_{d_k} \subset [0,1]} \left|  \int_{S_1 \times \cdots S_{d_k} } W_k(x_1, \cdots x_{d_k}) dx_1 \cdots dx_{d_k} \right|
\end{equation*}
and  for $\realgraphonspaceu$ the cut-distance 
\begin{equation*}
    \delta_\Box(W,V) = \inf_{\sigma \in \Sigma} \| W-V^\sigma \|_\Box 
\end{equation*}

\subsection{Subgraph density}

Let $F$ be a $(r,d)$-hypergraph and let $W$ a $(r,d)$-hypergraphon the subgraph density of $F$ in $W$ is computed by the formula 
\begin{equation*}
\label{eq:subgraphdensity}
    t(F,W) =  \int_{[0,1]^{|V(F)|}} \prod_{k=1}^r \prod_{(i_1, \cdots,i_{d_k}) \in E_k(F)} W(x_{i_1},\cdots,x_{i_{d_k}}) \prod_{i \in V(F)} dx_i 
\end{equation*}

\subsection{ Quantum graphs }

\begin{definition}[Quantum graphs]
 A quantum graph $ F$ is the   linear combination of a finite number of linear $(r,d)$-hypergraphs $F_i$ with real coefficients, more precisely  
\begin{equation*}
    F = \sum_i \alpha_i F_i \quad  \mbox{ and } \alpha_i \in \mathbb{R}.
\end{equation*}
The $(r,d)$-hypergraphs $F_i$ are the constituents of the quantum graph.
Hence  the definition of $t(F,W)$ extends to quantum graphs linearly, i.e. $t(F,W) = \sum_i \alpha_i t(F_i,W)$.
\end{definition}

\section{Step functions in $\realgraphonspacel$}
\label{sec:stepfunctionspace}

A  $(r,d)$-step function in $\realgraphonspacel$  is a function which is  constant in each box $(\frac{i_1-1}{m}, \frac{i_1}{m} ] \times \cdots \times (\frac{i_d-1}{m}, \frac{i_{d_k}}{m}]$ for each coordinate $k \in [r]$. A step function has a natural parameterization as follows: 

 Let $\lebesgue$ be the Lebesgue measure on $\mathbb{R}$.   Let $S=\{S_1, \cdots, S_m\}$ be a partition of $[0,1]$ where every $S_i$ is an interval,  let $\pi \in \convexcombm{m}$ such that $\nu(S_i)=\pi_i$. 
Let  $\mathbb{R}_S^{(m,r,d}$ be the set of $r$ dimensional vectors where each coordinate is a symmetric $d$-dimensional array with $m^d$ steps. 
Let $\convexcombm{m}_q$ be  the  $k$-simplex of probability vectors of length $m$ 
\begin{equation*}
\convexcombm{m} = \{ \pi \in [0,1]^m \mid \sum_i \pi_i = 1 \}.    
\end{equation*}
We call the elements $\pi \in \convexcombm{m}$ partition vectors of the step functions.
Let $\sfmat \in \mathbb{R}_S^{(m,r,d}$ and $\pi \in \convexcombm{m}$. Then a $m$-step function $(\sfmat,\pi):\prod_{k=1}^r ([0,1]^{d_k} \mapsto \mathbb{R}^r)$ is a function that locally constant on  $S_{i_1} \times \cdots \times S_{i_d}$ for all $i_1, \cdots, i_d \in [m]^d$ and in each coordinate $k \in [r]$, more precisely
\begin{equation*}
(\sfmat,\pi)_k =\sum_{i_1, \cdots, i_{d_k}=1}^m \sfmat_{i_1, \cdots, i_d, k} \mathbbm{1}_{S_{i_1}\times \cdots \times  S_{i_{d_k}}}.
\end{equation*}
We denote the space of  $m$-step functions by  $\realstepfunctionspacel{m}$. Thus $\realstepfunctionspacel{m}$  is identified by $\mathbb{R}_S^{(m,r,d)} \times P^{(m)}$.  It is convenient to define the class of step functions when $\pi$ is fixed. Hence $\realstepfunctionspacelpi{m}{\pi}=  \mathbb{R}_S^{(m,r,d} \times \{ \pi \}$. 
Let 
\begin{equation*}
n(m,r,d) = \sum_{k=1}^r \binom{m+d_k-1}{d_k}     
\end{equation*}
then the number of parameters to define a $m$-step function in $\realstepfunctionspacel{m}$ and $\realstepfunctionspacelpi{m}{\pi}$ are  $n(m,r,d)+m-1$ and  $n(m,r,d)$. In $\stepfunctionspacel{m}$, $t(F,W)$ is reduced to 
\begin{equation*}
t(F,(\sfmat,\pi)) = \sum_{\tiny x_1, \cdots, x_{|V(F)|}=1}^m \prod_{k=1}^r  \prod_{(i_1, \cdots, i_{d_k}) \in E_k(F)} \sfmat_{i_1,\cdots,i_{d_k},k } \prod_{i \in V(F)} \pi_i
\end{equation*}
Let $E(F) = \cup_k E_k(F)$. It sometimes convenient to write $t(F, (\sfmat,\pi))$ as follows,
\begin{equation*}
t(F,(\sfmat,\pi)) = \sum_{\tiny x_1, \cdots, x_{|V(F)|}=1}^m   \prod_{(i_1, \cdots, i_{d_k}) \in E(F)} \sfmat_{i_1,\cdots,i_d,k } \prod_{i \in V(F)} \pi_i
\end{equation*}

\subsection{Partial Derivatives of $t(F, (\sfmat,\pi))$}
 
In this section, we develop formulas for the partial derivative of the polynomials $t(F,(\sfmat,\pi))$. First, we define the notion of the labeled hypergraph and partial subgraph density. 
 \subsubsection{Labeled Graphs and Partial Subgraph Densities}
\begin{definition}[Labeled Graphs]
A k-labeled hypergraph is a $(r,d)-$hyper-graph $F$ to which we associate a vector of its vertices $(a_1\cdots a_k)\in (V(F))^k$ (for some $k\in\mathbb{N}$).  We denote this $k$-labeled graph by $F^{\bullet a_1 \cdots a_k}$. For a $k$-labeled graph, its $i$-th label is $a_i$ for $i\in[k]$. 
\end{definition}
We define the partial subgraph density of a $k$-labeled graph $F^{\bullet a_1\cdots a_k}$ as a $k$-dimensional function $t_{x_1, \cdots, x_k}(F^{\bullet  a_1 \cdots a_k},W):[0,1]^k \to [0,1]$ where
\begin{equation}
\label{eq:condsubgraphdensity}
t_{ x_1, \cdots,x_k}(F^{\bullet a_1 \cdots a_k},(\sfmat,\pi)) = \int_{[0,1]^{|F|-k}} \prod_{j=1}^r  \prod_{(i_1, \cdots, i_{d_j}) \in E_j(F)} \sfmat_{i_1,\cdots,i_{d_j},j } \prod_{i \in V \setminus [k]} dx_i.
\end{equation}
By definition, if $k=0$ then $t_{ x_1, \cdots,x_k}(F^{\bullet a_1 \cdots a_k},(\sfmat,\pi))$ is $t(F,(\sfmat,\pi))$. When considering step functions $(\sfmat,\pi)$, we will sometimes use $t_{\pi:i_1\cdots i_k}(F^{\bullet a_1\cdot a_k} , (\sfmat,\pi))$ as a shorthand for  $t_{x_1\cdots x_k}(F^{\bullet a_1\cdot a_k} , (\sfmat,\pi))$ where for $i\in[k]$ the value $x_i$ is in the $i$-th partition class, e.g., $x_i = \pi_i/2 + \sum_{j=1}^{i-1} \pi_j$.

Moreover, a notation for the linear combination of partial subgraph densities is also necessary. 
\begin{definition}[Linear combination of partial subgraph densities.]
  \label{def:partialbullets}
Let $F$ be a $(r,d)$-hypergraph and $k \in [r]$. We define a notation for a linear combination of partial subgraph densities  
\begin{equation*}
\partial_k^{(\bullet \cdots \bullet)} F =  \sum_{(a_1 \cdots a_{d_k}) \in E_k(F)} F^{\bullet(a_1 \cdots a_{d_k})}  
\end{equation*}
and
\begin{equation*}
    \partial^{ \bullet} F = \sum_{a \in V(F)} F^{\bullet a} .
\end{equation*}
where $F^{\bullet(a_1 \cdots a_{d_k})}$ denotes the $d_k$-labeled hypergraph obtained from $F$  by deleting the hyperedge $(a_1 \cdots a_{d_k})$ and  labeling the vertices $\{a_1, \cdots, a_{d_k}\}$ and $F^{\bullet a}$ denotes the $1$-labeled graph obtained from $F$  by labeling the node $a$, The partial subgraph densities of these labeled hypergraphs are
\begin{eqnarray}
\label{eq:labeledsubgraphdensity}
    && t_{i_1 \cdots i_{d_k} } (F^{\bullet(a_1 \cdots a_{d_k})},(\sfmat,\pi)) = \sum_{x_1, \cdots, x_{|F|-d_k}} \prod_{(s_1 \cdots s_{d_j})  \in E(F) \setminus \{a_1 \cdots q_{d_k}\}} \sfmat_{x_{s_1}, \cdots x_{d_j},j} \nonumber \\ && \prod_{k \in V(F) \setminus \{a_1, \cdots, a_{d_k}\}} \pi_{x_k}    
\end{eqnarray}
and
\begin{equation*}
    t_i(F^{\bullet a},(\sfmat,\pi)) = \sum_{x_1, \cdots, x_{|F|-1}} \prod_{(s_1 \cdots s_{d_k})  \in E_k(F) } \sfmat_{x_{s_1}, \cdots x_{d_k}} \prod_{k \in V(F) \setminus \{a\}} \pi_{x_k}    
\end{equation*}
\end{definition}

Note that $F^{\bullet(ab)}$ and $  F^{\bullet(ba)} $ are not always equivalent in the computation of the partial subgraph density.
Consider  $F^{\bullet(ab)}$ and $F^{\bullet(ba)}$ as shown in Figure  \ref{fig:labeledgraph}, then
\begin{equation*}
t_{uv} (F^{\bullet(ab)},W) = \int_{[0,1]} W(x,v)dx \quad \mbox{ and } \quad  t_{uv} (F^{\bullet(ba)},W) = \int_{[0,1]} W(u,y)dy.
\end{equation*}
\begin{figure}
    \centering
    \includegraphics[scale=0.8]{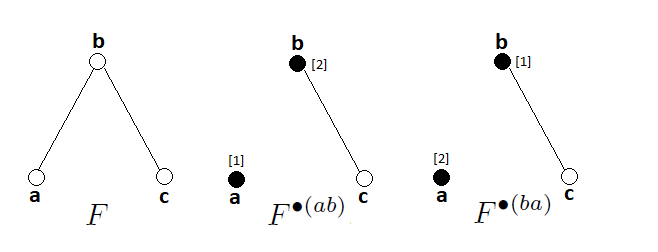}
    \caption{$F^{\bullet(ab)}$ and $F^{\bullet(ba)}$ are not equivalent. }
    \label{fig:labeledgraph}
\end{figure}

\subsubsection{Partial Derivatives of $t(F,)$ in $\realstepfunctionspacel{m}$}

We identify uniquely the partial derivative operator $\frac{\partial }{\partial \sfmat_{i_1 \cdots i_{d_k},k}}$ by the non-increasing sequence $i_1 \leq \cdots \leq i_{d_k}$.   $\frac{\partial {\sfmat}_{x_1 \cdots x_{d_j},j}}{\partial {\sfmat}_{i_1 \cdots i_{d_k},k}}$ is non-zero when $j=k$ and $\{x_1 \cdots x_{d_j}\} = \{ i_1 \cdots i_{d_k} \}$ thus 
\begin{equation*}
\label{eq:simplepartialderivative}
\frac{\partial {\sfmat}_{x_1 \cdots x_{d_j},j}}{\partial {\sfmat}_{i_1 \cdots i_{d_k},k}} = \delta_{j,k} \delta_{ \{x_1 \cdots x_{d_j}\} , \{ i_1 \cdots i_{d_k} \} } 
\end{equation*}
where $\delta_{i,j}$ is Kronecker's delta, and we overload  Kronecker's delta for finite sets, i.e. 
\begin{equation*}
    \delta_{A,B} = \begin{cases}
        1 & A = B \\
        0 & A \neq B
    \end{cases}
\end{equation*}
where $A$ and $B$ are finite sets.

\begin{restatable}{theorem}{thmpartialderivatives}
\label{thm:partialderivatives}
Let $(\sfmat,\pi) \in \realstepfunctionspacel{m}$ and let $orb((i_1, i_2, \cdots i_{d_k}))$ by the orbit of $\sfmat_{i_1, \cdots, i_{d_k}}$ the actions of $\Sigma_{d_k}$. 
Then the first partial derivative of $t(F,(\sfmat,\pi))$ are
\begin{eqnarray}
\label{eq:1partialderstep}
\frac{\partial t(F)}{\partial \sfmat_{i_1, \cdots, i_{d_k},k} } &=& \pi_{i_1} \cdots \pi_{i_{d_k}} \sum_{(j_1, \cdots,j_{d_k}) \in orb((i_1 \cdots i_{d_k})) } t_{j_1 \cdots j_{d_k}} (\partial_k^{ \bullet (\bullet \cdots \bullet)}F, (\sfmat,\pi) )
\end{eqnarray}
and
\begin{equation}
\label{eq:1partialderstep2}
\frac{\partial t(F, (A,\pi) )}{\partial \pi_i} = t_{i}(\partial^{\bullet} F,(A,\pi)) .
\end{equation}  

\end{restatable}

\begin{example}
Let $F=\Triangle$ and let $(\sfmat,\pi) \in \mathcal{W}^{(m,1,2)}_{\mathbb{R}}$. Then $t(\Triangle,(\sfmat,\pi))$ is computed by the formula
\begin{equation*}
 t(\Triangle,(\sfmat,\pi) = \sum_{i,j,k=1}^m \sfmat_{ij}\sfmat_{jk}\sfmat_{ki} \pi_i \pi_j \pi_k.
\end{equation*}
Then  $F^{\bullet (a,b)} = \ltriangled $ hence $\partial \Triangle = \ltrianglea + \ltriangleb + \ltrianglec$ and $ t_{ij}(\ltrianglea, (\sfmat,\pi)) = \sum_{k=1}^m \sfmat_{ik}\sfmat_{jk} \pi_k$. Hence the partial derivative $\frac{\partial t(\Triangle, (\sfmat,\pi))}{\partial A_{ij}}$ is computed by 
\begin{eqnarray*}
\frac{\partial t(\Triangle,(\sfmat,\pi))}{\partial A_{ij}} &&= \pi_i \pi_j \left( t_{ij}(\ltrianglea + \ltriangleb + \ltrianglec, (\sfmat,\pi) ) + (1-\delta_{ij}) t_{ji}(\ltrianglea + \ltriangleb + \ltrianglec, (\sfmat,\pi) ) \right) \\  &&= 3\pi_i \pi_j \left(  t_{ij}(\ltrianglea, (\sfmat,\pi) ) + (1-\delta_{ij}) t_{ji}(\ltrianglea, (\sfmat,\pi) ) \right) 
\end{eqnarray*}
By the symmetry of triangle, we have  $t_{ij}(\ltrianglea, (\sfmat,\pi))=t_{ij}(\ltriangleb, (\sfmat,\pi))=t_{ij}(\ltrianglec, (\sfmat,\pi))$.
then  we have 
\begin{eqnarray*}
\frac{\partial t(\Triangle,(\sfmat,\pi))}{\partial A_{ij}} &&= 3(2-\delta_{ij})\pi_i \pi_j   t_{ij}(\ltrianglea, (\sfmat,\pi) ) \\&& = 3(2-\delta_{ij})\pi_i \pi_j  \sum_{k=1}^m \sfmat_{ik}\sfmat_{jk} \pi_k
\end{eqnarray*}

\end{example}

\subsection{The split map}
On the space of step functions, we define the split of step functions,
\begin{definition}[Split of step functions]
\label{def:splitmap}
Let $\lambda \in (0,1)$   and let $k \in [m]$. we denote by $\theta(\cdot,\lambda,k):\realstepfunctionspacel{m} \to \realstepfunctionspacel{m+1}$ the map defined by $\theta((\sfmat,\pi), \lambda,k)=(\theta(\sfmat,\lambda,k),\theta(\pi,\lambda,k))$ where $\theta(\sfmat,\lambda,k)\in \mathbb{R}_S^{(m+1, r,d)}$ and $\theta(\pi,\lambda,k)\in \convexcombm{m+1}$ such that
\begin{equation*}
\theta_{i}(\pi,\lambda,k) = \left\lbrace
\begin{array}{c c}
\pi_{i} & i < k \\
\lambda \pi_{k} & i = k \\
(1-\lambda) \pi_{k} & i = k+1 \\
\pi_{i-1} &  k + 1 < i \leq m+1 
\end{array} \right.
\end{equation*}
and
\begin{equation*}
\theta_{i_1,\cdots,i_{d_k}, k }(\sfmat,\lambda,m) = \left\lbrace
\begin{array}{c c}
\sfmat_{i_1, \cdots,i_{d_k},k} & \forall_{n \in [d_k]  } i_n \in [k]  \\
\sfmat_{i_1, \cdots,i_{d-s} m \cdots m,k} & \exists_{s \in [d_k]  } i_{d_k-s+1} = i_{d_k-s+2} = \cdots = i_d = m+1  
\end{array}
\right. 
\end{equation*}
and $\theta_{i_1,\cdots,i_{d_k} }(\sfmat,\lambda,k)$ is obtained from $\theta_{i_1,\cdots,i_{d_k} }(\sfmat,\lambda,m)$ by a suitable permutation $\sigma \in \Sigma_m$ i.e. $\theta_{i_1,\cdots,i_{d_k} }(\sfmat,\lambda,k) = \theta_{i_1,\cdots,i_{d_k} }(\sfmat,\lambda,m)^\sigma$.

Note that the action of $\theta(\cdot, \lambda, k)$ on $(\sfmat,\pi)$ is the  split  the $k$-th row/column of $(\sfmat,\pi)$ in two rows/columns whose   weights are $\lambda \pi_k$ and $(1-\lambda) \pi_k$.
\end{definition}
\begin{figure}
    \centering
    \includegraphics[scale=0.6]{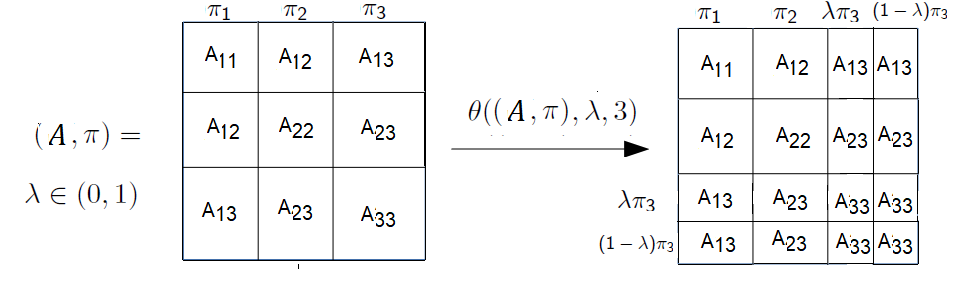}
    \caption{A split map transforms the representation of a $3$ step function into a $4$ step function}
    \label{fig:splitmap}
\end{figure}
The split map does not change $(\sfmat,\pi):[0,1]^d \to \mathbb{R}^r$ as function  only increases the dimensions to represent  $(\sfmat,\pi)$.
Figure \ref{fig:splitmap} shows how a split map transforms a bidimensional  $3$ step function into a $4$-step function, which is the same step function. 


\section{Uniform multi-relational hypergraphons}
\label{sec:uniformmulti-relationalhypergraphons}

Here we extend key results in Graphon Theory as Szemeredi's regularity lemma, Compactness, Counting Lemma, and Large Deviation principle from the space of one-relation hypergraphons, i.e., $\widetilde{\mathcal{W}}^{(1,d)}$ to   $\graphonspaceu$. To do that, we extend these key results for the case when $d$ is a vector of constant value. Hence, $\graphonspacel$ is the space of vectors of uniform multi-relational hypergraphons and $\graphonspaceu = \graphonspacel/\sim$. When $r=1$, Lubetzky and Zhao \cite{lubetzky2015replica}, extend these key results. Here, we extend the results when $r \geq 1$. 

\subsection{Key results in graphons}

\begin{theorem}[Szemeredi's regularity lemma]
\label{thm:Szemeredi1}
When $r=1$. 
For every $\epsilon >  0$ there exists some $M(\epsilon) > 0$ such that for every $W \in \graphonspacel$ there exist some $m \leq  M(\epsilon)$  and some $g \in \graphonspacel$  with $\| f -g \|_\Box \leq \epsilon $, and such that $g$ is constant in each box $(\frac{i_1-1}{m}, \frac{i_1}{m} ] \times \cdots \times (\frac{i_d-1}{m}, \frac{i_d}{m}]$.
\end{theorem}
\noindent Hence, we have the Szemeredi Regularity Lemma for $\graphonspacel$.
\begin{restatable}{theorem}{tmmSzezmeredi2}
\label{thm:Szemeredi2}
For every $\epsilon >  0$ there exists some $M(\epsilon) > 0$ such that for every $W \in \graphonspacel$ there exist some $m \leq  M(\epsilon)$  and some $g \in \graphonspacel$  with $\| f -g \|_\Box \leq \epsilon $, and such that g is
constant in each box $(\frac{i_1-1}{m}, \frac{i_1}{m} ] \times \cdots \times (\frac{i_d-1}{m}, \frac{i_d}{m}]$.    
\end{restatable}
\begin{proof}
    The proof is direct from Lemma $4.1$ in \cite{lovasz2007szeme}.
\end{proof}
\noindent Using Szemeredi's regularity lemma, we prove the compactness of $\graphonspaceu$.
\begin{restatable}{theorem}{tmmcompactness}
\label{thm:compactness}
($\graphonspaceu$, $\delta_\Box$) is compact. 
\end{restatable}
\begin{proof}
Based on Theorem \ref{thm:Szemeredi2}, the proof is identical to the compactness proof for classical graphons $\widetilde{\mathcal{W}}^{(1,2)}$ given in \cite{lovasz2006limits}.
\end{proof}
\noindent Another key result is the Counting Lemma.
\begin{lemma}[Counting Lemma ]
Let $F$ be a linear $(r,d)$-hypergraph then for all $V, W \in \graphonspaceu$ we have
\begin{equation*}
    |t(F,W) - t(F,V) | \leq \delta_\Box(W,V) 
\end{equation*}
\end{lemma}
\begin{proof}
    The proof is direct for the graph case given in \cite{borgs2008convergent}.
\end{proof}

\subsection{Large Deviation Principle for random hypergraphons}

To state properly the large deviation principle, we need some definition. 

\begin{definition}[Graphon representation of a $(r,d)$-hypergraph ]
\label{def:MRgraphonrepresentation}
Let $G$ be a $(r,d)$-hyper-graph. Then $f^G$ is the graphon representation of $G$ whose $k$ coordinate is defined by,
\begin{equation*}
f_k^{G}(x_1,\cdots, x_d) =  \begin{cases} 1 & \mbox{ if } (v_{\lceil nx_1 \rceil}, \cdots, v_{\lceil n x_d\rceil} )\in E_k(G)  \\
0 & \mbox{ otherwise}
\end{cases}
\end{equation*}
where $n = |V(G)|$.  
\end{definition}

\begin{definition}[Random hypergraph $G^{(r,d)}(n,p)$]
\label{def:randomhypergraph}
Let $n$ be a positive integer and $p \in [0,1]$. The Erd\H os-Renyi random $(r,d)$-hyper-graph is obtained by sampling each $k$ coordinate by taking every $d$-tuple of vertices from $[n]$ and connecting them with a $p$ probability.
\end{definition}

\noindent Let $\mathbb{P}^{(r,d)}_{n,p}$ be the probability distribution induced by $f^{G}$ where $G$ are  random $(r,d)$-hyper-graphs.

\begin{definition}[Rate function for $\mathbb{P}^{(r,d)}_{n,p}$]
Let $I^{(r,d)}: \graphonspaceu \to [0,\infty)$ be the function defined by,
\begin{equation*}
    I^{(r,d)}_p(W) = \sum_k \frac{1}{d_k!} \int_{[0,1]^{d_k}} h_{p}(W_k(x_1, \cdots,x_{d_k}) dx_1 \cdots dx_{d_k}
\end{equation*}
where 
\begin{equation*}
    h_{p}(x) = x \log (\frac{x}{p}) + (1-x) \log (\frac{1-x}{1-p}) 
\end{equation*}
\end{definition}

\begin{restatable}{theorem}{thmldp1}(Theorem $5.4$ in \cite{lubetzky2015replica})
\label{thm:LDP1}
When $r=1$.
For each fixed $p \in (0, 1)$, the sequence $\mathbb{P}_{n,p}$ obeys a large deviation principle in the space $\graphonspaceu$  with rate function hp. Explicitly, for any closed set $F \subseteq \graphonspaceu$
\begin{equation*}
    \limsup_{n \to \infty} \frac{1}{ (n)_{d_1}} \log \mathbb{P}^{(r,d)}_{n,p}(F) \leq -\inf_{W \in F} I^{(r,d)}_p( W) 
\end{equation*}
and for any open set $U \subseteq \graphonspaceu$ 
\begin{equation*}
    \liminf_{n \to \infty} \frac{1}{ (n)_{d_1} } \mathbb{P}^{(r,d)}_{n,p}(U)  \geq -\inf_{W \in U} I^{(r,d)}_p(W) 
\end{equation*}
where $(n)_x=n(n-1) \cdots (n-x+1)$.
\end{restatable}
To extend LDP for random $(r,d)$-hypergraphons when $d$ is uniform, we need the following Lemma
\begin{lemma}
\label{lem:auxldp}
When $r=1$. Let  $x \in \graphonspaceu$ and let $B(x, \eta) = \{ y \mid \delta_\Box(x,y) \leq \eta  \}$. Then for any $x \in \graphonspaceu$
\begin{equation*}
    \lim_{\eta \to 0} \limsup_{n \to \infty} \frac{1}{ (n)_{d_1} } \log \mathbb{P}_{n,p}^{(r,d)} (B(x,\eta) ) \leq I^{(r,d)}_p(x)
\end{equation*}
and for any arbitrary $\eta >0$
\begin{equation*}
    \liminf_{n \to \infty} \frac{1}{ (n)_{d_1} } \log \mathbb{P}_{n,p}^{(r,d)} (B(x,\eta) )\geq I^{(r,d)}_p(x)
\end{equation*}
\end{lemma}
\begin{proof}
The proof is identical to the proof Theorem $2.3 $ in \cite{chatterjee2011large}.
\end{proof}
Hence, we extend Theorem \ref{thm:LDP1} for the general case.
\begin{restatable}{theorem}{thmldp2}
\label{thm:LDP2}
For each fixed $p \in (0, 1)$, the sequence $\mathbb{P}_{n,p}$ obeys a large deviation principle in the space $\graphonspaceu$  with rate function hp. Explicitly, for any closed set $F \subseteq \graphonspaceu$
\begin{equation*}
    \limsup_{n \to \infty} \frac{r}{ (n)}_{d_1} \log \mathbb{P}^{(r,d)}_{n,p}(F) \leq -\inf_{W \in F} I^{(r,d)}_p( W) 
\end{equation*}
and for any open set $U \subseteq \graphonspaceu$ 
\begin{equation*}
    \liminf_{n \to \infty} \frac{r}{(n)_{d_1}} \mathbb{P}^{(r,d)}_{n,p}(U)  \geq -\inf_{W \in U} I^{(r,d)}_p(W) 
\end{equation*}
\end{restatable}
\begin{proof}
 Using Lemma \ref{lem:auxldp}, the proof is identical to Theorem $5$ in \cite{alvarado2022limits}   
\end{proof}

\section{Non-uniform multi-relational hyper-graphons}
\label{sec:nonuniformhypergraphons}

Here, we extend the results of the previous section to $\graphonspaceu$ when $d$ is a non-uniform vector.

Let $d$ be a vector of arities and $\max(d)$ the maximum arity. Note that $(\graphonspaceu, \delta) $ is a closed space of $(\graphonspaceuu{r}{\max(d)} , \delta_\Box)$. Hence, compactness and counting lemma hold in $(\graphonspaceu, \delta) $. 

To extend LDP, we first define the projection map $h$ and apply the contraction principle of Large Deviation Theory \cite{dembo2009large}. Let $h:\graphonspaceuu{r}{\max(d)} \to \graphonspaceuu{r}{d}$ be the map defined by 
\begin{equation}
h(W_k(x_1,\cdots, x_{\max(d)} )) = W_k(x_1,\cdots, x_{d_k}, x_{d_k}, \cdots, x_{d_k}  )
\end{equation}
for all $k \in [r]$. This map forgets for each $k$ hypergraphon the last $max(d)-d_k$ coordinates. In this setting, a $(r,1)$-hypergraphon is the limit of growing hypergraphs with isolated vertices and self-loops.
\begin{lemma}
\label{lem:contractionmap}
The map $h:\graphonspaceuu{r}{\max(d)} \to \graphonspaceuu{r}{d}$ is continuous in $(\graphonspaceuu{r}{\max(d)}, \delta_\Box)$.    
\end{lemma}
\begin{proof}
It is sufficient to prove that $\| h(W) \|_\Box \leq \|W \|_\Box$. Indeed, it is clear that the inequality holds
\begin{eqnarray*}
&&    \sup_{S_1, \cdots S_{d_k} \subseteq [0,1]} \left|\int_{S_1 \times \cdots S_{d_k} } W(x_1, \cdots, x_{d_k}) dx_1 \cdots dx_{d_k} \right| \leq \\
&& \sup_{S_1, \cdots S_{\max(d)} \subseteq [0,1]} \left|\int_{S_1 \times \cdots S_{\max(d)} } W(x_1, \cdots, x_{\max(d)}) dx_1 \cdots dx_{\max(d)} \right|    
\end{eqnarray*}

\end{proof}

\subsection{Large deviation principle for random $(r,d)$-hypergraphs}
We prove the LDP principle for $\graphonspaceuu{r}{d}$. To do that, we apply the contraction principle in Large Deviation Theory.

\noindent Let $\mathbb{P}^{(r,d)}_{k,n,p}$ be the probability distribution induced by $f_k^{G}$ where $G$ are  random $(r,d)$-hypergraaphs.

\begin{theorem}[Contraction Principle Theorem $4.2.1$ in \cite{dembo2009large}]
\label{thm:contractionprinciple}    
Let $\mathcal{X}$ and  $\mathcal{Y}$  be Hausdorff topological spaces and $f:\mathcal{X} \to \mathcal{Y}$ a continuous function. Consider a good rate\footnote{A good rate function is a rate function for which all the level sets $I([0, \alpha])$ are compact subsets of $\mathcal{X}$.} function $I :\mathcal{X} \to [0,\infty] $.
\begin{enumerate}
    \item For each $y \in \mathcal{Y}$, define
\begin{equation*}
I'(y) = \inf \{ I(x) \mid x \in \mathcal{X}, y=f(x)\}    
\end{equation*}
Then $I'$  is a good rate function on $\mathcal{Y}$, where as usual the infimum over the
empty set is taken as $\infty$.

\item If $I$ controls the LDP associated with a family of probability measures
$\mathbb{P}$ on $\mathcal{X}$, then $I'$ controls the LDP associated with the family of probability measures $\mathbb{P} \circ f^{-1}$ on $\mathcal{Y}$.
\end{enumerate}
\end{theorem}

\noindent A consequence of the above lemma is $\graphonspaceu$ is compact when $d$ is a vector of arities. Now, we prove LDP for random multi-relational hypergraphs with variable arities.
\begin{restatable}{theorem}{thmldp3}
\label{thm:LDP3}
Let $d$ be a $r$-dimensional vector of arities. 
For each fixed $p \in (0, 1)$, the sequence $\mathbb{P}_{n,p}$ obeys a large deviation principle in the space $\graphonspaceu$  with rate function hp. Explicitly, for any closed set $F \subseteq \graphonspaceu$
\begin{equation*}
    \limsup_{n \to \infty} \sum_k \frac{1}{(n)_{d_k} } \log \mathbb{P}^{(r,d)}_{k,n,p}(F) \leq -\inf_{W \in F} I^{(r,d)}_p( W) 
\end{equation*}
and for any open set $U \subseteq \graphonspaceu$ 
\begin{equation*}
    \liminf_{n \to \infty} \sum_k \frac{1}{ (n)_{d_k}} \mathbb{P}^{(r,d)}_{k,n,p}(U)  \geq -\inf_{W \in U} I^{(r,d)}_p(W) 
\end{equation*}
where $\mathbb{P}^{(r,d)}_{k,n,p}$ is the probability distribution induced by $G_k^{(r,d_k)}(n,p)$.
\end{restatable}
\begin{proof}
The proof is direct from Theorem \ref{thm:contractionprinciple} and Lemma  \ref{lem:contractionmap}.   
\end{proof}

\begin{definition}[Constrained region defined by quantum graphs]
Let $\mathcal{F}$ be an ordered set  of quantum $(r,d)$-graphs  $\mathcal{F} = [F_1, \cdots,F_k]$  and let $u \in \mathbb{R}^k$  be a vector of multi-relational subgraph densities. Then we define the constrained region defined by $(\mathcal{F},u)$
\begin{equation}
  \label{eq:feasibleregion}
  \feasibleregion=\{ W \in \graphonspaceu \mid \wedge_i t(F_i,W)=u_i\{.    
\end{equation}

\end{definition}
\noindent Hence, we prove the most typical random $(r,d)$-hyper-graphs of $\feasibleregion$.

\begin{theorem} 
\label{thm:mosttypicalrandommultigraph}
  Let $\mathcal{F}$ be an ordered set of  quantum $(r,d)$-graphs  $\mathcal{F} = [F_1, \cdots,F_k]$  and let $u \in \mathbb{R}^k$  be a vector of multi-relational subgraph densities. Then the limits $W^*(\mathcal{F},u)$ of the sequence $(G^{(r,d)}(n,1/2))_{n=1}^\infty$ of growing random $(r,d)$-hypergraphs which are uniformly sampled such that $\lim_{n \rightarrow \infty} G^{[r]}(n,1/2) \in   \feasibleregion $ satisfy 
\begin{equation}
W^*(\mathcal{F},u)   = \argmin_{W \in \feasibleregion} I_{1/2}^{(r,d)}(W) 
\end{equation}
\end{theorem}
\begin{proof}
 The proof is identical to Theorem $3.1$ in \cite{chatterjee2011large}.   
\end{proof}
\noindent Hereafter we write $I_{1/2}^{(r,d)}$ as $I$ and $h_{1/2}$ as $I_0$.

\section{Modeling Possible Worlds in Probabilistic Logic}
\label{sec:possibleworlds}

Let $L=(\forall, \wedge, \Rightarrow, R_1, \cdots, R_r)$ be a signature of a first-order language with only symmetric and relational symbols. If we have an ordered set of logical closed formulas $\mathcal{F}$ from $L$ and vector $u \in (0,1)^{|\mathcal{F}|}$ then we can form a probabilistic relational theory $(\mathcal{F}),u$ by attaching each logical formula $\mathcal{F}_i$  with $u_i$. 

Let $L=(\forall, \wedge, \Rightarrow, Friends; Sm)$ be a relational signature with the predicates Friends and Smoker(Sm). From this signature, we build the following probabilistic theory   
\begin{eqnarray*}
 &&   u_1: \forall_x Sm(x) \\
 &&   u_2: \forall_{x,y} Friends(x,y) \\
 &&   u_3: \forall_{x,y,z} Friends(x,y) \wedge Friends(y,z) \wedge Friends(z,x) \\
 &&   u_4: \forall_{x,y} Friends(x,y) \Rightarrow (Sm(x) \Leftrightarrow Sm(y))
\end{eqnarray*}
From this theory, we can make probabilistic logical inferences. For example
What is the probability of the following formula
\begin{eqnarray*}
     \forall_{x,y,z}  Friends(x,y) \wedge Friends(y,z) \wedge Friends(z,x) \\\Rightarrow (Sm(x)) \wedge Sm(y) \wedge Sm(z))    
\end{eqnarray*}
is true?

From Theorem \ref{thm:mosttypicalrandommultigraph}, The most typical random $(2,[2,1])$-hyper-graphs $\optimalsolution$ are the solution of the following optimization problem.
\begin{equation}
\label{eq:wexample}
    \min_{[Friends, Sm] \in \widetilde{\mathcal{W}}^{(r,d)}} \frac{1}{2} \int_{[0,1]^2} I_{0}(Friends(x,y)) dxdy  + \int_{[0,1]}  I_0(Sm(x)) dx
\end{equation}
subjected to 
\begin{eqnarray*}
&&    u_1 = \int_{[0,1]} Sm(x) dx \\
&&    u_2 = \int_{[0,1]^2} Friends(x,y) dx dy \\
&&    u_3 = \int_{[0,1]^3} Friends(x,y)Friends(y,z) Friends(z,x)  dx dy dz \\
\end{eqnarray*}
For the fourth constraint, we transform the formula to the equivalent one 
\begin{equation}
\label{eq:Pexample}
 \forall_{x,y} Friends(x,y) \Rightarrow ( (Sm(x) \Rightarrow S(y)) \wedge (Sm(y) \Rightarrow S(x))     
\end{equation}
and using the formula $p \Rightarrow q \equiv 1 - (p \wedge \neg q$)
\begin{eqnarray*}
    && 1-u_4 = \int_{[0,1]^2} Friends(x,y) \\ && ( 1- (1-Sm(x)(1-Sm(y))) (1- Sm(y)(1-Sm(x)))  dxdy.
\end{eqnarray*}
The last constraint can be reduced to a constraint of a quantum graph.

Let $W^*(\mathcal{F},u)$ be the set of solutions of (\ref{eq:wexample}). Then, the probability of (\ref{eq:Pexample}) is computed by 
\begin{multline*}
    1 - \frac{1}{|W^*(\mathcal{F},u)|} \sum_{[Friend,Sz] \in W^*(\mathcal{F},u)} \int_{[0,1]^3} Friend(x,y)Friend(y,z)Friend(z,x) \\  (1-Sz(x))(1-Sz(y))(1-Sz(z)) dxdydz.    
\end{multline*}

\section{Computability of $\optimalsolution$}
\label{sec:computability}

$\optimalsolution$ is defined as the solution of an optimization problem in an infinity dimensional space. Hence, it is not clear whether $\optimalsolution$ is computable. Here we show $\optimalsolution$  is computable in almost all cases when $(\mathcal{F},u)$  is feasible.


\subsection{Marginal Polytope}

\begin{definition}[Marginal map]
Let $\mathcal{F}$ be an ordered vector of quantum graphs. Then the marginal map $t(\mathcal{F},\cdot):\realgraphonspacel \to \mathbb{R}^{|\mathcal{F}|}$ is defined by 
\begin{equation*}
    t(\mathcal{F},W) = (t(\mathcal{F}_1,W), \cdots, t(\mathcal{F}_k,W)  ).
\end{equation*}
We denote by $J_x(\mathcal{F})$ the Jacobian of $t(\mathcal{F},\cdot)$.
\end{definition}

\begin{definition}[Marginal Polytope]
The marginal polytope $\marginalpolytopem{m}$ of  $\mathcal{F}$ is the image of $t(\mathcal{F},\cdot)$  on 
$\stepfunctionspacel{m}$, i.e., $\marginalpolytopem{m}= t(\mathcal{F},\stepfunctionspacel{m})$. 
We also define the total marginal polytopes of $T(\mathcal{F})=t(\mathcal{F},  \graphonspacel)$ and $T(\mathcal{F},\pi )=t(\mathcal{F},  \stepfunctionspacelpi{m}{\pi})$.  
\end{definition}

\begin{definition}[Constrained regions]
If $u \in T(\mathcal{F})$, the level sets of  $t(\mathcal{F}, u)$ define regions constrained by subgraph densities. Hence 
\begin{equation*}
    \realfeasibleregionl = t^{-1}(\mathcal{F},u) 
\end{equation*}
and 
\begin{equation*}
\realfeasibleregionml{m} = \realfeasibleregionl \cap \realstepfunctionspacel{m} 
\end{equation*}
and if we fix $\pi \in \convexcombm{m}$
\begin{equation*}
\realfeasibleregionmlpi{m}{\pi} =  \realfeasibleregionml{m} \cap \realstepfunctionspacelpi{m}{\pi}  
\end{equation*}    
\end{definition}

\subsubsection{Non-empty interior of $T^{(m)}(\mathcal{F})$}
To guarantee that $\realfeasibleregionml{m}$ are analytical manifolds for almost all $u \in T^{(m)}(\mathcal{F})$ by applying Theorem \ref{thm:sardstheorem} (Sard's Theorem), we need a sufficient condition for which $T^{(m)}(\mathcal{F})$ has a non-empty interior. Hence, it is sufficient that $\mathcal{F}$ is a set of independent quantum graphs.  
\begin{definition}[Independent graphs]
Let $\mathcal{F}$ be an ordered and finite set of quantum graphs.  We say that  $\mathcal{F}$ is a set of independent graphs if for all $m$ such that  $n(m,r,d)  \geq |\mathcal{F}|$  there is   $(\sfmat,\pi) \in \realstepfunctionspacel{m} $ such that $\pi \in (\convexcombm{m})^\circ$ and  $J_{(\sfmat,\pi)}(\mathcal{F})$ is full rank\footnote{This criterion is equivalent to say the polynomials $t(\mathcal{F}_i, X)$ $i \in [|\mathcal{F}|]$ are algebraically independent.}.
\end{definition}
The following lemma shows that $T^{(m)}(\mathcal{F})$ and $T^{(m)}(\mathcal{F},\pi)$ have a non-empty interior when $\mathcal{F}$ is a set of independent graphs. 
\begin{restatable}{lemma}{lemopenmarginal}
\label{lem:openmarginal}
The following statements are equivalent.
\begin{enumerate}
\item  $\mathcal{F}$ is a set of independent  graphs and  $n(m,r,d) \geq |\mathcal{F}| $.
\item $T^{(m)}(\mathcal{F}, \pi )$ has a non-empty interior for all $\pi \in (\convexcombm{m})^\circ$. 
\item $t(\mathcal{F},U)$ has a non-empty interior for any non-empty open set $U \subset \stepfunctionspacel{m} $.
\end{enumerate}
\end{restatable}


\begin{definition}[Density function]
\label{def;densitiy}
Let $f_0:\mathbb{R} \to \mathbb{R}$ be an analytic function. Then $f_s:\realgraphonspacel \to \mathbb{R}$  is defined by 
\begin{equation*}
f_s(W)= \sum_{k=1}^r \int_{[0,1]^{d_k}} f_0(W_k(x_1,\cdots, x_{d_k})) dx_1  \cdots dx_{d_k} 
\end{equation*} 
\end{definition}
\noindent On $\realstepfunctionspacel{m}$   the density function $f_s$ has the form
\begin{equation}
\label{eq:generalratefunctionstep}
f_s((\sfmat,\pi))= \sum_{k=1}^r \sum_{x_1 \cdots x_k=1}^m f_0(\sfmat_{x_1, \cdots x_{d_k},k}) \pi_{x_1} \cdots.\pi_{x_{d_k}} 
\end{equation} 
 
\noindent Hence, we need to prove the solutions  
\begin{equation*}
    \optimalsolution = \argmin_{W \in \feasibleregion} I(W)
\end{equation*}
are step functions. To prove that, we prove the solutions 
\begin{equation*}
     \optimalsolutionf{f_s} = \argmin_{W \in \realfeasibleregionl} f_s(W)
\end{equation*}
are step functions for any density function $f_s$.  Note that the solutions of the above optimization problems are labeled graphons. 
However, the orbits of these solutions are unlabeled graphons since the subgraph density constraints and the density functions are invariant under the measure-preserving map. The statement that $\realoptimalsolution{f_s} $ are step functions here is called the Principle of Optimization of Density functions $(POD)$. The computability of $\optimalsolution$ is a special case of this principle.

\subsection{The principle of optimization of density functions}
This principle is defined in the following theorem 

\begin{definition}
Let $\mathcal{F}$ be a set of quantum graphs and let $u \in T(\mathcal{F})$. Then, the base size of step functions is,   
\begin{equation*}
    m_0(\mathcal{F}, u ) = \min \{ m | n(m,r,d)  > |\mathcal{F}| \mbox{ and } u \in T(\mathcal{F}) \} 
\end{equation*}
\end{definition}

 \begin{restatable}{theorem}{thmPOD}
 \label{thm:POD}
 Let $\mathcal{F}$ be a set of independent graphs and let $m \geq m_0(\mathcal{F},u)$.
 For almost all cases $u \in \marginalpolytopem{m_0} \cap \Omega(\mathcal{F})$  the  minima of $f_s$ on $\realfeasibleregionl$ lies in $\realfeasibleregionml{m_0}$.
 \end{restatable}

\noindent The approach to prove POD is to approximate $\realoptimalsolution{f_s}$ by a sequence of solutions from the step function spaces $\realoptimalsolutionm{m}{f_s}$ for increasing $m$, defined by
\begin{equation*}
    \realoptimalsolutionm{m}{f_s}  = \argmin_{W \in \realfeasibleregionml{m}} f_s(W),
\end{equation*}
The above optimization is on finite-dimensional space, and the idea is to prove 
the split $\realoptimalsolutionm{m}{f_s}$  is in $ \realoptimalsolutionm{m+1}{f_s}$ and there is no new solution in 
 \begin{equation*}
 \realoptimalsolutionm{m+1}{f_s} \setminus \theta( \realoptimalsolutionm{m}{f_s} ,\lambda, k)    
\end{equation*}
for all $k \in [m]$. To prove Theorem  \ref{thm:POD}, we need develop results for  $\realfeasibleregionmlpi{m}{\pi}$ and $\realfeasibleregionml{m}$,

Figure \ref{fig:theoremdependencies} shows the theorem dependencies to prove the computability of $\optimalsolutionf{f_s}$. There are three cornerstones to prove POD. The first one is Theorem \ref{thm:numbercomponents}, which shows the number of connected components of $\realfeasibleregionml{m+1}$ is not higher than the number of connected components of $\realfeasibleregionml{m}$.  The second cornerstone is Theorem \ref{thm:splitminimum} that ensures the split of local minima of $f_s$ in $\realfeasibleregionml{m}$ is again a local minimum if $\realfeasibleregionml{m+1}$. The third cornerstone is Theorem \ref{thm:almostsmoothnessCr}, which states the split of critical points of $f_s$ on $\realfeasibleregionml{m}$ does not bifurcate into other critical points for almost all cases $u \in T(\mathcal{F})$.
With these three theorems, we prove in Theorem \ref{thm:localminimumsplit} that any local minimum of $f_s$ in $\realfeasibleregionml{m+1}$  is the split of a local minimum of $f_s$ in $\realfeasibleregionml{m}$. Theorem \ref{thm:POD} proves POD by the closure of Theorem \ref{thm:localminimumsplit}. Lemma \ref{lem:randomsolutions}   shows that if $u$ is not an extremal value, then $\optimalsolutionmf{m}{I}$ has values only in $(0,1)$.  Finally,  Theorem \ref{thm:radin} proves the computability of $\optimalsolution$.

\begin{figure}
    \centering
\tikzstyle{block} = [draw, rectangle, 
    minimum height=3em, minimum width=6em]
\tikzstyle{sum} = [draw, circle, node distance=1cm]
\tikzstyle{input} = [coordinate]
\tikzstyle{output} = [coordinate]
\tikzstyle{pinstyle} = [pin edge={to-,thin,black}]

\begin{tikzpicture}[auto, node distance=2cm]


\node [block] (connected) at ( 0,4) {$\begin{array}{c}
     \mbox{No new connected component}   \\
     \mbox{Theorem } \ref{thm:numbercomponents}
\end{array}$};

\node [block] (nobifurcation) at ( 0,2) {$\begin{array}{c}
     \mbox{No bifurcation of critical points}   \\
     \mbox{Theorem } \ref{thm:almostsmoothnessCr}
\end{array}$};

\node [block, below of=nobifurcation] (localminima) {$\begin{array}{c}
     \mbox{Split of local minima}   \\
     \mbox{Theorem } \ref{thm:splitminimum}
\end{array}$};

\node [block] (inductionsplit) at ( 6,2) {$\begin{array}{c}
     \mbox{Inductive local minima}  \\
     \mbox{Theorem } \ref{thm:localminimumsplit} 
\end{array}$};

\node [block] (pod) at ( 6, 0) {$\begin{array}{c}
     \mbox{POD}  \\
     \mbox{Theorem } \ref{thm:POD} 
\end{array}$};

\node [block] (nonextremal) at ( 0,-2) {$\begin{array}{c}
     \mbox{Non Extremal Statistics  }   \\
     \mbox{ Lemma }\ref{lem:randomsolutions}  
\end{array}$};

\node [block] (radin) at ( 6,-2) {$\begin{array}{c}
     \mbox{Computability of $\optimalsolution$ proof}   \\
     \mbox{Theorem } \ref{thm:radin}
\end{array}$};

 \draw [->] (connected) --  (inductionsplit);
\draw [->] (localminima) --  (inductionsplit);
\draw [->] (nobifurcation) --  (inductionsplit);
\draw [->] (inductionsplit) --  (pod);
\draw [->] (pod) --  (radin);
\draw [->] (nonextremal) --  (radin);

\end{tikzpicture}

    \caption{Theorem dependencies to prove the computability of $\optimalsolution$}
    \label{fig:theoremdependencies}
\end{figure}
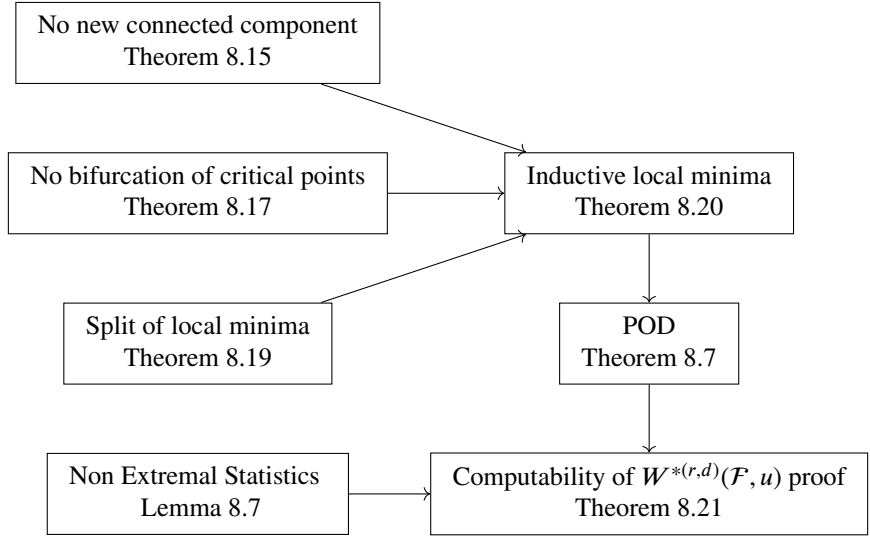

\subsection{Smoothness of $\realfeasibleregionml{m}$}
\label{sec:smoothness}
Here, we develop sufficient conditions to ensure $\realfeasibleregionml{m}$ is an analytical manifold. 
\begin{definition}
\label{def:criticalvalues}
Let $F:M \to N$ be a smooth function.  A point $c \in N$ is said to be a regular value of $F$ if every point of the level set $F^{-1}(c)$ is a regular point; otherwise, it is a critical value. 
\end{definition}

\begin{definition}[Regular and Critical Values of $t(\mathcal{F}, \cdot)$]
\label{def:regularvalue}
 Let $Cri^{(m)}(\mathcal{F})$ be the set of critical values of the map  $t(\mathcal{F},\cdot): \realstepfunctionspacel{m}\to \mathbb{R}^{|\mathcal{F}|}$. Then, the set of regular values is defined by 
\begin{equation}
\label{eq:regularvalues}
\reg{\mathcal{F}} = \cup_{m >0} \marginalpolytopem{m} \setminus \cup_{m >0}  Cri^{(m)}(\mathcal{F})    
\end{equation}
and 
\begin{equation*}
 \regstatsm{m}{\mathcal{F}}= \marginalpolytopem{m} \cap \reg{\mathcal{F}}.
\end{equation*}
It is convenient to define the set of non-extremal values of $\mathcal{F}$ defined by  
\begin{equation*}
 \Omega(\mathcal{F})= \reg{\mathcal{F}} \setminus\cup_{m>0} \partial T^{(m)}(\mathcal{F})   
\end{equation*}

\end{definition}

\noindent Note that $\cup_{m >0}  Cri^{(m)}(\mathcal{F})$ has zero measure since it is the countable union of zero measure sets. Thus if $T^{(m)}(\mathcal{F})$ has a non-empty interior then  $Reg^{(m)}(\mathcal{F})$ and $Reg(\mathcal{F})$ are dense in $T^{(m)}(\mathcal{F})$ and $T(\mathcal{F})$. If $u \in Reg(\mathcal{F})$ then $\realfeasibleregionml{m}  \subset  t^{-1}(\mathcal{F},u) \subset \realstepfunctionspacel{m}$ has only regular points for hence $\realfeasibleregionml{m}$ is an analytical manifold  of $n(m,,r,d)+m-1 - |\mathcal{F}|$ dimensions since the marginal map is polynomial. 

To ensure that  $u \in T^{(m)}(\mathcal{F}) $ is a regular value for almost all cases, we apply Sard's theorem in \cite{lee2003smooth}
\begin{theorem}[Theorem $6.10$ in \cite{lee2003smooth} (Sard's Theorem)]
\label{thm:sardstheorem}
Suppose $M$ and N are smooth manifolds with or without boundary, and $F: M \to N$ is a smooth map. Then, the set of critical values of $F$ has zero measure in $N$.
\end{theorem}
Hence if $T^{(m)}(\mathcal{F})$ has a non-empty interior then from Sard's theorem $u \in T^{(m)}(\mathcal{F})$ is a regular value for almost everywhere in $T^{(m)}(\mathcal{F})$. In other words, if $\mathcal{F}$ is a set of independent graphs and we pick $u \in Reg(\mathcal{F})$  at random, then the probability that $\realfeasibleregionml{m}$ is a smooth manifold is $1$. Hereafter, we assume that $\mathcal{F}$ is a set of independent graphs.

\subsection{The smoothness of $\realfeasibleregionmlpi{m}{\pi}$ almost everywhere}

$\realfeasibleregionmlpi{m}{\pi}$ is a real algebraic variety\footnote{A real algebraic variety is a subset of $\mathbb{R}^n$ defined by polynomial equations.}, and every real algebraic variety $V$ can be decomposed on smooth manifolds using the so-called Whitney Stratification \cite{kaloshin2005geometric}.  The difference between real algebraic varieties and smooth manifolds is subtle. Real algebraic varieties can have non-uniform dimensions; when a real algebraic variety has a uniform dimension, it is a smooth manifold. However, we are interested in seeing $\realfeasibleregionmlpi{m}{\pi}$  as a smooth manifold and investigating when it is smooth or smooth almost everywhere. Here, we use the same ideas of Whitney Stratification to prove that $\realfeasibleregionmlpi{m}{\pi}$ is almost everywhere but not necessarily with a uniform dimension. First we prove that  $\realfeasibleregionmllpi{m}{ [F] }{ [u]  }{  \pi  }$ is a smooth manifold.

\begin{restatable}{lemma}{lemnonzerogradient}
\label{lem:nonzerogradient}
Let $F$ be a simple graph. Then 
\begin{equation*}
\langle (\sfmat,\pi), \nabla t(F, (\sfmat,\pi) \rangle = |E(F)| t(F,(\sfmat,\pi)) 
\end{equation*}
any $(\sfmat, \pi)$ and $|E(F)|$ is the number of edges in $F$.   
\end{restatable}

\begin{restatable}{lemma}{lemsmoothnessalmosteverywhere}
\label{lem:smoothnessalmosteverywhere}
Let $F=\sum_k a_k F_k$ be a quantum graph.
If $u \neq 0$  and $\pi \in (\convexcombm{m})^\circ$ then 
$\realfeasibleregionmllpi{m}{ [\mathcal{F} }{ [u]  }{  \pi  }$ is smooth almost everywhere.
\end{restatable}
Hence  $\realfeasibleregionmllpi{m}{ [F] }{ [u]  }{  \pi  }$ is an analytical manifold almost everywhere for any quantum graph $F$, $u \neq 0$ and $\pi \in \convexcombm{m}$.  Hence, we prove that  $\realfeasibleregionmlpi{m}{\pi}$ is analytical almost everywhere. 

\begin{restatable}{theorem}{thmsmoothnessalmosteverywhere}
\label{thm:smoothnessalmosteverywhere}
Let $\mathcal{F}$ be a set of independent graphs. If $u \in T^{(m)}(\mathcal{F})^\circ $  then   $\realfeasibleregionmlpi{m}{\pi}$ is analytical  almost everywhere for all $\pi \in (\convexcombm{m})^\circ$.
\end{restatable}

\begin{proof}
The proof is based on the induction of the number of constraints. The base step is ensured by Lemma \ref{lem:smoothnessalmosteverywhere}.

The induction step. $\realfeasibleregionmllpi{m}{ [\mathcal{F}_1, \cdots \mathcal{F}_n]  }{ [u_1, \cdots, u_n]  }{  \pi  }$ has regular points almost everywhere and  $\realfeasibleregionmllpi{m}{ [\mathcal{F}_1, \cdots \mathcal{F}_n]  }{ [u_1, \cdots, u_n]  }{  \pi  }$  can have different connected components and for each connected component $N$  of $d_N$ dimensions, we have two possibilities 

\begin{enumerate}
    \item The   critical points of $t(\mathcal{F}_{n+1}, \cdot )$ on $\realfeasibleregionmllpi{m}{ [\mathcal{F}_1, \cdots \mathcal{F}_n]  }{ [u_1, \cdots, u_n]  }{  \pi  } $ have a non-empty interior. Hence   $t(\mathcal{F}_{n+1}, \cdot )$ has critical points to the entire $N$. In this case, $t(\mathcal{F}_{n+1}, \cdot )$ has dependency with $[\mathcal{F}_1, \cdots \mathcal{F}_n]$, thus the constraint $t(\mathcal{F}_{n+1},(\sfmat,\pi))=u_{n+1}$  is redundant. It follows, $N_1 = N \cap \realfeasibleregionmllpi{m}{ [\mathcal{F}_{n+1}] }{ [u_{n+1}]  }{  \pi  }$ is analytical almost everywhere  and $N_1$ has  $d_N$ dimensions.

    \item The critical points of $t(\mathcal{F}_{n+1}, \cdot )$ on $\realfeasibleregionmllpi{m}{ [\mathcal{F}_1, \cdots \mathcal{F}_n]  }{ [u_1, \cdots, u_n]  }{  \pi  } $ have an empty interior. It follows, $N_1 = N \cap \realfeasibleregionmllpi{m}{ [\mathcal{F}_{n+1}] }{ [u_{n+1}]  }{  \pi  }$ is analytical almost everywhere  and $N_1$ has  $d_N-1$ dimensions.
\end{enumerate}
Hence, we conclude that  $\realfeasibleregionmlpi{m}{\pi}$ is analytical almost everywhere.
\end{proof}

\subsection{Non-increasing number of connected components of $\realfeasibleregionml{m}$}

One cornerstone of the main result is proof that there is no new connected component in $\realfeasibleregionml{m+1} \setminus \realfeasibleregionml{m+1}$. Hence, the number of connected components of $\realfeasibleregionml{m}$ is not increased when $m$ is increased to $m+1$. To prove this, we need two definitions.

\begin{definition}
Let $m \leq m'$. Let $T^{-(m,m')}$ be a set of the level sets of $t(\mathcal{F}, \cdot):\realstepfunctionspacel{m'} \to \mathbb{R}^{|\mathcal{F}|}$ when $u \in T^{(m)}(\mathcal{F})$. More precisely,

\begin{equation*}
    T^{-(m,m')}(\mathcal{F}) = \{ x \in \realstepfunctionspacel{m'} \mid t(\mathcal{F},x) \in T^{(m)}(\mathcal{F}) \}.
\end{equation*}
Hence it is clear that $t(\mathcal{F}, T^{-(m,m')}(\mathcal{F})) = T^{(m)}(\mathcal{F})$ for all $m \leq m'$.
\end{definition}

\begin{definition}[The Aggregated Feasible Region]
Let $M^{(m)}(\mathcal{F}) \subset \mathbb{R}^{n(m,r,d)+m-1+|\mathcal{F}|}$  be the aggregated region of $\realfeasibleregionml{m}$ defined by 
\begin{equation*}
    M^{(m)}(\mathcal{F}) = \{ (x,z) \in \realstepfunctionspacel{m} \times \mathbb{R}^{|\mathcal{F}|}: \forall i  \in [|\mathcal{F}|] \quad t(\mathcal{F}_i,x)-z_i = 0 \}
\end{equation*}
\end{definition}
\begin{remark}
\label{rem:aggregatedfeasibleregion}
Note that $M^{(m)}(\mathcal{F})$ is the graph of the marginal map, thus $M^{(m)}(\mathcal{F})$  is connected.  Hence, the regular  and critical points of the marginal map  are arbitrarily close in 
$\realfeasibleregionml{m}$ when $u$ varies in $T^{(m)}(\mathcal{F})$.

\end{remark}

\begin{restatable}{theorem}{thmnumbercomponents}
\label{thm:numbercomponents}
Let $\mathcal{F}$ be a set of independent graphs and let $u \in \Omega(\mathcal{F})$ and let $m \geq m_0(\mathcal{F},u)$. Then there is no new connected component $N \subset \realfeasibleregionml{m+1} \setminus \realfeasibleregionml{m}$. Hence 
the  number of connected components of $\realfeasibleregionml{m+1}$  is not higher than the number of connected components of $\realfeasibleregionml{m}$.

\end{restatable}

\subsection{The space of critical points of $f_s(\cdot)$ on $\realfeasibleregionml{m}$ }
Recall that a point $(\sfmat,\pi) \in \realfeasibleregionml{m}$
 is a  critical point of $f(\cdot,\pi)$ if there is 
  $\beta \in \mathbb{R}^{|\mathcal{F}|}$ such that
\begin{equation*}
\nabla f_s((\sfmat,\pi)) =  \sum_{i=1}^{|\mathcal{F}|} \beta_i \nabla t(\mathcal{F}_i,(\sfmat,\pi))
\end{equation*}
Here we study the critical points of all  $\realfeasibleregionml{m}$  when $u$ ranges over all the feasible combinations of subgraph densities in $T^{(m)}(\mathcal{F})$.  Then, the spaces of critical points are defined by
\begin{definition}
\label{def:criticalmanifold}
Let  $\mathcal{F}$ be a set of independent simple graphs.  Let $f_s(\cdot)$ be a smooth function defined in  (\ref{eq:generalratefunctionstep}) and let $m$ be a positive integer such that $n(m,r,d) > |\mathcal{F}|$ then 
\begin{eqnarray*}
\criticalpoints{m} &=& \{ ((\sfmat,\pi),\beta) \in \realstepfunctionspacel{m} \times \mathbb{R}^{|\mathcal{F}|} \mid  t(\mathcal{F},(\sfmat,\pi)) \in Reg(\mathcal{F}) \\ && \nabla f_s  -  \sum_{i=1}^{|\mathcal{F}|} \beta_i \nabla t(\mathcal{F}_i)\big{|}_{(\sfmat,\pi)}=0\}.
\end{eqnarray*}
\end{definition}

\subsubsection{The split map on critical points of $f_s$ }
\label{sec:splitmapcritical}

Here we prove that the split map $\theta(\cdot, \lambda, k):\realstepfunctionspacel{m} \to \realstepfunctionspacel{m+1} $ preserves the critical points of density functions $f_s$ i.e. $\theta( \criticalpoints{m}, \lambda, k) \subset \criticalpoints{m+1} $.

\begin{restatable}{lemma}{lemcriticalsplit}
\label{lem:criticalsplit}
The split operation preserve the critical points of $f_s$ either on $\realfeasibleregionmlpi{m}{\pi^*}$ or on $\realfeasibleregionml{m}$.
\begin{enumerate}
\item Let  $(\sfmat^*,\pi^*) \in \realfeasibleregionml{m}$ be a critical point  of $f_s$. Then $\theta((\sfmat^*,\pi^*),\lambda,k)$ is a critical point of $f_s$ on $\realfeasibleregionml{m+1}$ with the same Lagrange multipliers for any  $\lambda \in [0,1]$ and $ k \in [m]$.

    \item Let  $(\sfmat^*,\pi^*) \in \realfeasibleregionmlpi{m}{\pi^*}$ be a critical point  of $f_s$. Let $c \in \mathbb{R}$. Let $g_\epsilon: \realstepfunctionspacel{m+1} \to \mathbb{R}$, defined by 
\begin{equation*}
    g_c((\sfmat, \pi)) = c \sum_{ij=1}^{m+1} (\sfmat_{ij} - \theta_{ij}( \sfmat^*, 0, k) )^2 \pi_i \pi_j 
\end{equation*}
and let $f_c = f_s + g_c$. Recall that $\theta( \sfmat^*, 0, k)$ is a matrix  with duplicate row/columns at  $k$ and  $k+1$ and  $\theta_{ij}( \sfmat^*, 0, k)$ is the $ij$ entry. 

Then $\theta((\sfmat^*,\pi^*),\lambda,k)$ is a critical point of $f_c$ on $\realfeasibleregionmlpi{m+1}{\theta(\pi^*,\lambda,k)}$ with the same Lagrange multipliers for any $c \in \mathbb{R}$ and $\lambda \in [0,1]$ and $ k \in [m]$.

\end{enumerate}
\end{restatable}

\subsubsection{Smoothness of $\criticalpoints{m}$ }

A sufficient condition to prove the split operation on critical points does not bifurcate into another critical point is the smoothness of $\criticalpoints{m}$. Hence, we prove $\criticalpoints{m}$ is smooth almost everywhere. Thus, the split does not bifurcate other critical points almost cases when $u \in T(\mathcal{F})$.

\begin{restatable}{theorem}{thmalmostsmoothnessCr}
\label{thm:almostsmoothnessCr}    
The split of any critical point in $\criticalpoints{m+1}$ does not bifurcate into another critical point for almost all cases when $u \in T(\mathcal{F})$ and $m \geq m_0$.
\end{restatable}

\subsubsection{The split map on minimal points of $f_s$ }


We have seen that the split map preserves the critical points of $f_s$. Here, we prove that the split map preserves local minima of $f_s$. Let $(\sfmat^*,\pi^*)$ be a local minimum of $f_s$ on $\realfeasibleregionmlpi{m}{\pi^*}$.
From Theorem \ref{thm:smoothnessalmosteverywhere}, $\realfeasibleregionmlpi{m}{\pi^*}$ is smooth  almost everywhere.  Hence here we assume that $(\sfmat^*,\pi^*)$ is not necessarily a regular point in $\realfeasibleregionmlpi{m}{\pi^*}$. Thus the tangent space on $(\sfmat^*,\pi^*)$ is not necessarily defined. Still, it is defined the tangent cone   $C_{(\sfmat^*,\pi^*)} \realfeasibleregionmlpi{m}{\pi^*}$  which generalizes the tangent space.

First, we consider a perturbation of $f_s$ by the quadratic distance  from $(\sfmat^*,\pi^*)$. 

\begin{definition}[Perturbation of $f_s$ at the local minimum $(\sfmat^*,\pi^*)$]
\label{def:fdperturbation}
Let  $(\sfmat^*,\pi^*) \in \realstepfunctionspacel{m+1} $ and  
and $\epsilon >0$. Then $f_\epsilon = f_s + g_\epsilon$  where  $g_\epsilon: \realstepfunctionspacel{m+1} \to \mathbb{R}$ is defined by 
\begin{equation*}
    g_\epsilon((A, \pi)) = \epsilon \sum_{kl=1}^{m+1} (A_{kl} - \sfmat^*_{kl})^2 \pi_k \pi_l. 
\end{equation*}
\end{definition}
Note that if $(\sfmat^*,\pi^*)$  is a local minimum of $f_s$ on $\realfeasibleregionmlpi{m}{\pi^*}$ then $(\sfmat^*,\pi^*)$   is a local minimum of $f_\epsilon$.

Recall from Theorem \ref{thm:necessarysecondordercondition}, if  $(\sfmat^*,\pi^*)$  is a local minimum of $f_s$ on  $\realfeasibleregionmlpi{m}{\pi^*}$ then  $H^g_{(\sfmat^*,\pi^*)} f_s$  is positive semidefinite on $C_{(\sfmat^*,\pi^*)} \realfeasibleregionmlpi{m}{\pi^*}$ and  $H^g_{(\sfmat^*,\pi^*)} f_\epsilon$  is positive definite on $C_{(\sfmat^*,\pi^*)} \realfeasibleregionmlpi{m}{\pi^*}$.

 Now we prove the following theorem that shows the preservation of local minima of $f_\epsilon$ on $\realfeasibleregionml{m}$  by the split map for sufficiently small $\lambda>0$.

Thus it is clear that if $\criticalpoints{m}$ is smooth then $t(\mathcal{F}, \cdot): \criticalpoints{m} \to \mathbb{R}^{|\mathcal{F}|}$ is a 

\begin{restatable}{lemma}{lemsplitminimum}
\label{lem:splitminimum}
Let $f_\epsilon$ be a perturbed $f_s$ given in Definition \ref{def:fdperturbation}.
Let $\mathcal{F}$ be a set of independent graphs and let $u \in T(\mathcal{F})$ and let $m \geq m_0(\mathcal{F},u)$. Let  $(\sfmat^*,\pi^*) \in \realfeasibleregionml{m}$ be a local minimum of $f_\epsilon$ on $\realfeasibleregionml{m}$. Then for a given $\epsilon>0$ and $k \in [m]$ there is a $\lambda_0 > 0$ such that  $\theta((\sfmat^*,\pi^*),\lambda_0,k) $ is a local minimum of $f_\epsilon$ on  $\realfeasibleregionml{m}$.
\end{restatable}

The following theorem extends the conclusion of the above lemma from a $\lambda_0 >0$ to  $\lambda_0 \in (0,1)$.

\begin{restatable}{theorem}{thmsplitminimum}
\label{thm:splitminimum}
Let $\mathcal{F}$ be a set of independent graphs and let $u \in T(\mathcal{F})$ and let $m \geq m_0(\mathcal{F},u)$. Let  $(\sfmat^*,\pi^*) \in \realfeasibleregionml{m}$ be a local minimum of $f_s$ on $\realfeasibleregionml{m}$. Then $\theta((\sfmat^*,\pi^*),\lambda,k) \in \realfeasibleregionml{m+1}$ is a local minimum of $f_s$  on $\realfeasibleregionml{m}$.  for any $\lambda \in [0,1]$ and  any $k \in [m]$. 

\end{restatable}
The next theorem proves that every global minimum in $\realoptimalsolutionm{m+1}{f_s} $  is the split of a local minimum on $\realfeasibleregionml{m}$.
\begin{restatable}{theorem}{thmlocalminimumsplit}
\label{thm:localminimumsplit}
Let $\mathcal{F}$ be a set of independent graphs and let $u \in T(\mathcal{F})$ and let $m \geq m_0(\mathcal{F},u)$.  For almost all cases   $u \in \marginalpolytopem{m_0} \cap \Omega(\mathcal{F})$ we have
\begin{equation*}
    \bigcup_{ \begin{array}{c}
      \lambda \in [0,1]     \\
       k \in [m]   
    \end{array}  } \theta( \realoptimalsolutionm{m}{f_s}, \lambda,k ) = \realoptimalsolutionm{m+1}{f_s}.
\end{equation*}
\end{restatable}

\begin{proof}
From Theorem \ref{thm:splitminimum}, the split of any local minimum of $f_s$ on $\realfeasibleregionml{m}$ is a local minimum on  $\realfeasibleregionml{m+1}$.  
Hence, it is sufficient to prove that there is no global minimum on $\realfeasibleregionml{m+1}$ different from the split of a global minimum on $\realfeasibleregionml{m}$. Note that any local minimum can only be created by the bifurcation or by creating a new connected component in $\realfeasibleregionml{m+1}$.  From Theorem  \ref{thm:almostsmoothnessCr}, the split operation does not bifurcate into another critical point for all cases $u \in T(\mathcal{F})$.
 There is no new connected component from Theorem \ref{thm:numbercomponents}, $\realfeasibleregionml{m+1}$. Hence, we conclude that any global minimum of $f_s$ on $\realfeasibleregionml{m+1}$ is the split of a global minimum of $f_s$ on $\realfeasibleregionml{m}$.
\end{proof}

\subsection{Proof of the principle of optimization of density functions}
\label{sec:mainresult}

From Theorem \ref{thm:localminimumsplit} and by induction, we conclude $\realoptimalsolutionm{m_0}{f_s} \subseteq \realoptimalsolutionm{m}{f_s}$.  To prove POD, it is necessary to prove $\realoptimalsolutionm{m}{f_s} \setminus \realoptimalsolutionm{m_0}{f_s}$ is empty. First,  we need to prove that $f_s$ and $I$ are continuous functions in $L^{1}([0,1]^2)$ topology
\begin{restatable}{lemma}{lemcontinuity}
\label{lem:continuity}
\quad
\begin{enumerate}
    \item Let $f_s:\realgraphonspacel \to \mathbb{R}$ be the function defined by 
    \begin{equation*}
        f_s(W) = \sum_{k=1}^r \int_{[0,1]^{d_k}} f_0(W_k(x_1, \cdots, x_{d_k})) dx_1 \cdots dx_{d_k}
    \end{equation*}
    where $f_0:\mathbb{R} \to \mathbb{R}$ is smooth.. Then $f_s$ is continuous in $\realgraphonspacel$ in $L^{1}( \prod_k [0,1]^{d_k})$ topology. 
    \item Let $I:\graphonspacel \to \mathbb{R}$ be the function defined by 
    \begin{equation*}
        I(W) = \sum_{k=1}^r \int_{[0,1]^{d_k}} I_0(W_k(x_1, \cdots, x_{d_k})) dx_1 \cdots dx_{d_k}
    \end{equation*}
    where $I_0(x)= x\log(x) + (1-x) \log(1-x)$. Then $I$ is continuous  in $L^{1}( \prod_k [0,1]^{d_k})$  topology when $|W| < 1$. 
\end{enumerate}

\end{restatable}
Here, we have all the ingredients to prove POD.

\thmPOD*

Thus to compute a global minimum $\optimalsolutionf{f_s}$, it is necessary to compute $m_0(\mathcal{F},u)$ and then to compute a global minimum $W^*$ of $f_s$ on $\realfeasibleregionml{m_0(\mathcal{F},u)}$.  

\subsubsection{The case when $f_s$ is the rate function $I$}
\label{sec:caseI}

Note that if $f_s =I$, then $f_s$ is not entirely defined on $\realgraphonspaceu$ but if $H_x I$ is continuous at any local minimum of $\realstepfunctionspacel{m}$,  then Theorem \ref{thm:POD} can be applied. The following lemma warranties that $H_x I$ is continuous at any local minimum.

\begin{restatable}{lemma}{lemrandomsolutions}
\label{lem:randomsolutions}
Let $\mathcal{F}$ be a set of independent graphs and let $u \in \Omega(\mathcal{F})$ and let $m \geq m_0(\mathcal{F},u)$.
Then every local minimum of $I$  in $\feasibleregionml{m}$ is in $\randfeasibleregionml{m}$. 
\end{restatable}
Finally, we prove the computability of $\optimalsolution$. 
\begin{restatable}{theorem}{thmradin}
\label{thm:radin}
Let $\mathcal{F}$ be a set of independent graphs, then every global minimum of $I$ on $\feasibleregion$ is a $m_0(\mathcal{F},u)$ step function for all almost cases when  $m \geq m_0(\mathcal{F},u)$ and  $u \in \Omega(\mathcal{F})$. 
\end{restatable}

\begin{proof}
From Lemma \ref{lem:randomsolutions}, every local minimum  $\optimalsolutionf{I}$ is in    $\randfeasibleregionml{m_0}$; hence $I$ is a smooth function at any local minimum. Locally, $I$ is a function defined on $\realstepfunctionspacel{m}$. Thus, the conditions to apply  Theorem \ref{thm:POD} hold. 

\end{proof}

\DeclareRobustCommand{\edgetriangle}
{%
W^*([\begin{tikzpicture}
   \draw (0,0) [fill=white] circle [radius=0.05] --(0,0.2) [fill=white] circle [radius=0.05]    ;  
   \end{tikzpicture},\begin{tikzpicture}
    \draw (0,0) [fill=white] circle [radius=0.05] --(0.1,0.2) [fill=white] circle [radius=0.05] --  (0.2,0) [fill=white] circle [radius=0.05] -- (0,0) ;      
    \end{tikzpicture}],[\rho, \tau])
}

\section{ Concluding Remarks and Open Problems }
\label{sec:conclusions}

That $\optimalsolution$ is a step function means that the most typical infinite and undirected networks have infinitely many redundancies. The existence of infinitely many redundancies is not new in complex networks. Bianconi et al.\ in  \cite{bianconi2001bose} have shown the existence of Bose-Condensate on infinite networks using other assumptions.  These redundancies can be seen as `` data patterns'' inside of these types of networks. 

These redundancies can also be seen as a partition of individuals (vertices) where each part represents a type of individual.  This conclusion could answer the question: Why, if we are free to make our own decisions, do we have a similar lifestyle to our neighbors and friends? On one side, the entropic forces can be seen as forces to increase our possibilities of expression. In other words, entropic forces are forces pursuing individual freedom. On the other hand, we have to follow or respect our socioeconomic constraints. Therefore, we have to be aligned to one type of lifestyle to maximize freedom and, at the same time, satisfy socioeconomic constraints.  

This work opens some problems.

 \subsection{Computational complexity to find a $\optimalsolution$}
 
 For Statistical Relational Learning $(SLR)$, this a natural problem because the complexity to compute $t(\mathcal{C},\optimalsolution)$ is $O(m_0(\mathcal{F},u)^{|E(\mathcal{C})|}))$ and $SLR$ is interested in having probabilistic inference methods with polynomial computational complexity. 

\DeclareRobustCommand{\Wedgetriangle}
{%
W^{*(1,2)}([\begin{tikzpicture}
   \draw (0,0) [fill=white] circle [radius=0.05] --(0,0.2) [fill=white] circle [radius=0.05]    ;  
   \end{tikzpicture},\begin{tikzpicture}
    \draw (0,0) [fill=white] circle [radius=0.05] --(0.1,0.2) [fill=white] circle [radius=0.05] --  (0.2,0) [fill=white] circle [radius=0.05] -- (0,0) ;      
    \end{tikzpicture}],[\rho, \tau])
}

 Radin et al.\ \cite{radin2013phase} found  the solution $\Wedgetriangle$  when $0 < \rho  < 0.5$ and $0 < \tau < \rho^3$ has a closed form given by 
 \begin{equation}
 \label{eq:trianglesolution}
\optimalsolutionmf{1,2}{I}(x,y) = \begin{cases}
    \rho - z & x,y < 1/2 \mbox{ or }  x,y > 1/2 \\
    \rho + z &  x < \frac{1}{2} < y \mbox{ or } y < \frac{1}{2} < x,
\end{cases}
\end{equation}
and there is only one global optimum. 
Hence, the computational complexity to obtain $\optimalsolution$ is polynomial. Still, there can be other combinations of $(\mathcal{F},u)$, which might be several global/local minima. Hence, the computation might be an NP-complete problem. Thus it is an open problem  to understand when $(\mathcal{F},u)$ yields 
$\optimalsolution$ obtained with polynomial complexity.  
 
 We conjecture that a sufficient condition to compute $\optimalsolution$ in a polynomial time is $\optimalsolution$  is only one global/local minimum
since it is possible to devise a gradient descent algorithm that improves every step of an initial solution until the end of $\optimalsolution$.

\subsection{Uniqueness of $\optimalsolution$}

The uniqueness problem of  $\optimalsolution$ is not easy and is directly related to finite forcible graphons. A graphon $W$ is finite forcible \cite{lovasz2011finitely} if there is a finite number of subgraph densities constraints of simple graphs $(\mathcal{F},u)$ such that $W$ is uniquely determined by $(\mathcal{F},u)$. 
For our problem, $(\mathcal{F},u)$ is finite forcible if $\optimalsolution$ is unique. Lov\'asz conjectured that Finite Forcibility is true, but recently, Grzesik et al.\ in \cite{grzesik2018elusive} have shown that the extremal solutions of $t(F, \cdot)$ or $I(\cdot)$ on $\feasibleregion$ are not necessarily unique thus Finite Forcibility is false. Hence, it is an open problem to show for which combination of $(\mathcal{F},u)$, $\optimalsolution$ is unique.

\subsection{Learning of implicit constraints of big networks}

In the current literature, there are methods to estimate a Stochastic Block Model $W$ from a graph $G$ such that $\mathbb{G}(n,W)$ is sufficiently closed to $G$, and $n$ is the number of nodes of $G$. Examples of these methods are given in \cite{wolfe2013nonparametric}, \cite{gao2015rate} and \cite{klopp2017optimal}.

These methods allow learning implicit constraints $(\mathcal{F},u)$ of an observed network $G$. The idea is to assume that $G$ is the most typical random graph from $\mathbb{G}(n, \optimalsolution)$ where $(\mathcal{F},u)$ is unknown. 

Let $W(G)$ be a Stochastic Block Model (step function) estimated from $G$. Hence, we assume $W(G) = \optimalsolution$; hence, $W(G)$ is a critical point.  Thus there is $\beta \in \mathbb{R}^{|\mathcal{F}|}$ such that 
\begin{equation*}
    \nabla I(W(G)) = \sum_{i=1}^{|\mathcal{F}|} \beta_i \nabla t(\mathcal{F}_i,W(G) )
\end{equation*}
for some unknown $\mathcal{F}$. Hence given  $\mathcal{F}$, the problem to find $\beta$ becomes a linear regression problem, i.e.

\begin{equation}
\label{eq:inverseproblem}
    \min_\beta \| \nabla I(W(G)) - \sum_{i=1}^{|\mathcal{F}|} \beta_i \nabla t(\mathcal{F}_i,W(G) ) \|_2. 
\end{equation}
Thus to find the constraints $(\mathcal{F},u)$ is equivalent to find the best set of variables $\{\nabla t(\mathcal{F}_i,.)\}$ from the infinite set of simple graphs $\mathcal{G}$ such that  
\begin{equation*}
\| \nabla I(W(G)) - \sum_{i=1}^{|\mathcal{F}|} \beta_i \nabla t(\mathcal{F}_i,W(G) ) \|_2 < \epsilon \quad  \mbox{ and } \quad \| W(G) - W^*(\mathcal{F}, u)\|_1 < \epsilon   
\end{equation*}
 where  $\epsilon$ is a given error and $t(\mathcal{F}, W(G))=u$. Thus, the problem is given an $\epsilon > 0$ combinatorial search of a finite subset $\mathcal{F}$ of $\mathcal{G}$, the set of all simple and finite graphs.  This problem can seen as an unbounded version of subset selection in logistic regression \cite{miller2002subset}.

\section*{Acknowledgments}
The authors would like to thank Marco Zambon, who has had the patience and kindness to review this paper countless times. This work is partially supported by  ERC-StG 240186 MiGraNT: Mining Graphs and Networks, a theory-based approach, and SENESCYT of Ecuador.

\bibliographystyle{plain}
\bibliography{references}

\section*{Proofs}

\thmpartialderivatives*

\begin{proof}
In the next lines, we abbreviate $t(F,(\sfmat, \pi))$ as $t(F)$. For the first partial derivative, we compute
\begin{eqnarray*}
\frac{\partial t(F)}{\partial \sfmat_{i_1, \cdots, i_{d_k},k}} &=& \sum_{x_1, \cdots,x_{|F|}=1}^m   \frac{\partial }{\partial \sfmat_{i_1 \cdots i_{d_k},k}} \left\lbrace  \prod_{(s_1, \cdots,s_{d_j}) \in E(F)} \sfmat_{x_{s_1} \cdots x_{s_{d_j}},j} \right\rbrace \prod_{s \in V(F)} \pi_{x_s} \\
&=& \sum_{x_1, \cdots,x_{|F|}=1}^m \left\lbrace  \sum_{(a_1 \cdots a_{d_u}) \in E(F)} \frac{\partial \sfmat_{x_{a_1} \cdots x_{a_{d_u},u}} }{\partial \sfmat_{i_1 \cdots i_{d_k},k}}  \prod_{(s_1 \cdots s_{d_j}) \in E(F^{(a_1 \cdots a_{d_u} )})} \sfmat_{x_{s_1} \cdots x_{s_{d_j}},j} \right\rbrace \\ && \prod_{s \in V(F)} \pi_{x_s} \\
\end{eqnarray*}
By plugging (\ref{eq:simplepartialderivative}) into the above formula, we have 
\begin{eqnarray*}
\frac{\partial t(F)}{\partial \sfmat_{i_1, \cdots, i_{d_k},k}} &=&  \sum_{x_1, \cdots,x_{|F|}=1}^m \left\lbrace  \sum_{(a_1 \cdots a_{d_u}) \in E(F)} \delta_{u,k} \delta_{ \{x_{a_1} \cdots x_{a_{d_u}}\} , \{ i_1 \cdots i_{d_k} \} }   \right. \\&& \left. \prod_{(s_1 \cdots s_{d_j}) \in E(F^{(a_1 \cdots a_{d_u} )})} \sfmat_{x_{s_1} \cdots x_{s_{d_j}},j} \right\rbrace  \prod_{s \in V(F)} \pi_{x_s}   
\end{eqnarray*}
Then 
\begin{eqnarray*}
\frac{\partial t(F)}{\partial \sfmat_{i_1, \cdots, i_{d_k},k}} &=&  \sum_{(a_1 \cdots a_{d_k}) \in E_k(F)}  \sum_{x_1, \cdots,x_{|F|}=1}^m   \delta_{ \{x_{a_1} \cdots x_{a_{d_k}}\} , \{ i_1 \cdots i_{d_k} \} }    \\&&  \prod_{(s_1 \cdots s_{d_j}) \in E(F^{(a_1 \cdots a_{d_k} )})} \sfmat_{x_{s_1} \cdots x_{s_{d_j}},j}   \prod_{s \in V(F)} \pi_{x_s} \\    
\end{eqnarray*}
Note that $\delta_{ \{x_{a_1} \cdots x_{a_{d_k}}\} , \{ i_1 \cdots i_{d_k} \} } $ is $1$ only in orbit of $\sfmat_{{i_1} \cdots i_{d_k}}$ by the actions of $\Sigma_{d_k}$. Let $orb((i_1, i_2, \cdots i_{d_k}))$ be orbits of $\sfmat_{i_1 \cdots i_{d_k}}$ by the actions of $\Sigma_{d_k}$ then
\begin{eqnarray*}
\frac{\partial t(F)}{\partial \sfmat_{i_1, \cdots, i_{d_k},k} } &=& \sum_{(a_1 \cdots a_{d_k}) \in E_k(F)} \sum_{x_1, \cdots,x_{|F|-d_k}=1}^m  \sum_{(j_1, \cdots,j_{d_k}) \in orb((i_1 \cdots i_{d_k})) }  \delta_{ \{x_{a_1} \cdots x_{a_{d_k}}\} , \{ i_1 \cdots i_{d_k} \} }    \\&&  \prod_{(s_1 \cdots s_{d_j}) \in E(F^{(a_1 \cdots a_{d_k} )})} \sfmat_{x_{s_1} \cdots x_{s_{d_j}},j}   \prod_{s \in V(F)} \pi_{x_s}     
\end{eqnarray*}
Hence 
\begin{eqnarray*}
\frac{\partial t(F)}{\partial \sfmat_{i_1, \cdots, i_{d_k},k} } &=& \pi_{i_1} \cdots \pi_{i_{d_k}} \sum_{(j_1, \cdots,j_{d_k}) \in orb((i_1 \cdots i_{d_k})) }  \sum_{(a_1 \cdots a_{d_k}) \in E_k(F)}  \\&&  \sum_{x_1, \cdots,x_{|F|-d_k}=1}^m   \prod_{(s_1 \cdots s_{d_j}) \in E(F^{(a_1 \cdots a_{d_k} )})} \sfmat_{x_{s_1} \cdots x_{s_{d_j}},j}   \prod_{s \in V(F)} \pi_{x_s}     
\end{eqnarray*}
Thus
\begin{eqnarray*}
\frac{\partial t(F)}{\partial \sfmat_{i_1, \cdots, i_{d_k},k} } &=& \pi_{i_1} \cdots \pi_{i_{d_k}} \sum_{(j_1, \cdots,j_{d_k}) \in orb((i_1 \cdots i_{d_k})) } t_{j_1 \cdots j_{d_k}} (\partial^{ \bullet (a_1 \cdots a_{d_k})}, (\sfmat,\pi) )
\end{eqnarray*}
and 
\begin{eqnarray*}
\frac{\partial t(F)}{\partial \pi_i} &=& \sum_{x_1, \cdots, x_{|F|=1}}^m \prod_{(st) \in E(F)} \sfmat_{x_s x_t}  \frac{\partial } {\partial \pi_i} \left(\prod_{s \in V(F) } \pi_{x_s} \right) \\
&=& \sum_{x_1, \cdots, x_{|F|=1}}^m  \prod_{(st) \in E(F)} \sfmat_{x_s x_t}   \sum_{a \in V(F)} \frac{\partial \pi_{x_a} } {\partial \pi_i} \prod_{s \in V(F) \setminus \{a\} } \pi_{x_s}  \\
&=& \sum_{a \in V(F)} t_i (F^{\bullet a},(A,\pi)) \\
&=& t_{i}(\partial^{\bullet} F,(A,\pi)) .
\end{eqnarray*}

\end{proof}

\lemopenmarginal*
\begin{proof}
$(1) \implies (2)$. Let $\pi' \in (\convexcombm{m})^\circ$ be the partition vector such that $J_{(\sfmat',\pi')}(\mathcal{F})$ for some $\sfmat'$.

It is sufficient to prove that for any in  $\pi \in (\convexcombm{m})^\circ$, there is $\sfmat \in \mathbb{R}^{(m,r,d)}_S$ such that $(\sfmat,\pi)$ is a regular point of the marginal map $t(\mathcal{F},\cdot)$ restricted to $\stepfunctionspacelpi{m}{\pi}$.

Let $\left[  [z_{i_1 \cdots i_{d_1}} ]_{i_1 \cdots i_{d_1} =1}^m \cdots [z_{i_1 \cdots i_{d_r}} ]_{i_1 \cdots i_{d_r} =1}^m \right]$ and $[y_i]_{i=1}^{m}$ be symmetric arrays of real variables whose domain ranges $\realstepfunctionspacel{m}$.  Recall formula (\ref{eq:1partialderstep})

\begin{eqnarray*}
\frac{\partial t(F)}{\partial \sfmat_{i_1, \cdots, i_{d_k},k} } &=& \pi_{i_1} \cdots \pi_{i_{d_k}} \sum_{(j_1, \cdots,j_{d_k}) \in orb((i_1 \cdots i_{d_k})) } t_{j_1 \cdots j_{d_k}} (\partial^{ \bullet (a_1 \cdots a_{d_k})}, (\sfmat,\pi) ),
\end{eqnarray*}
Let $r(z,y)= \frac{\partial t(F, (\sfmat,\pi) )}{\partial {\sfmat}_{i,j}}|_{A=z, \pi = y}$. We observe that $r(z,y)$ is a homogeneous polynomial of grade $|V(F)|+|E(F)|-1$ on variables $z$ and $y$. 

Let $G(\sfmat,\pi) = J_{(\sfmat,\pi)}(\mathcal{F})^\top J_{(\sfmat,\pi)}(\mathcal{F})$ be the gramian matrix of the Jacobian matrix of marginal map $t(\mathcal{F},\cdot ):\stepfunctionspacelpi{m}{\pi^{(m)}} \to \mathbb{R}^{|\mathcal{F}|}$. Recall also that $G(\sfmat,\pi)$ is square and $G(\sfmat,\pi)$ is full rank iff $J_{(\sfmat,\pi)}(\mathcal{F})$ is full rank. 

Let $Q(z,y)=det(G(\sfmat,\pi)|_{A=z, \pi=y})$. Hence $Q(z,y)$ is a homogeneous polynomial of grade $(|V(F)|+|E(F)|-1)^{2|\mathcal{F}|}$ on variables $z$ and $y$. Note that every monomial $q(z,y)$ of $Q(z,y)$ is equal to $r(y)s(z)$ where $r$  and $s$ are homogeneous polynomials as well. 

Since $\mathcal{F}$ is a set of independent graphs, $Q(z,\pi')$ is a non-zero polynomial.
Hence we can obtain another polynomial $Q(z,\pi)$ from $Q(z,\pi')$ by substituting  each monomial $q(z,\pi')=r(\pi')s(z)$  of $Q(z,\pi')$ by $q(z,\pi)=r(\pi)s(z)$. Since $\pi \in (\convexcombm{m})^\circ$, $r(\pi)$ of each monomial $q(z,\pi)$ is non zero. Hence, $Q(z,\pi)$ is a non-zero polynomial.  
Therefore there is $\sfmat \in \mathbb{R}_S^{(m,r,d)}$ such that $Q(\sfmat,\pi) \neq 0$. It follows that $(\sfmat,\pi)$ is regular point of  the marginal map $t(\mathcal{F},\cdot)$ restricted to $ \stepfunctionspacelpi{m}{\pi}$.

$(2) \implies (3)$. Since $T^{(m)}(\mathcal{F},\pi)$ has a non empty interior in $\mathbb{R}^{|\mathcal{F}|}$, $J_x(\mathcal{F})$ is full rank almost everywhere in $ \stepfunctionspacelpi{m}{\pi}$ due to Sard's theorem and the analyticity of $t(\mathcal{F},\cdot)$. Then $J_x(\mathcal{F})$ is full rank almost everywhere in $ \stepfunctionspacel{m}$. Hence, for any non-empty open set $U \subset \stepfunctionspacel{m}$, $t(\mathcal{F},U)$ has a non-empty interior.   

$(3) \implies (1)$. Now we assume that $T^{(m)}(\mathcal{F},\pi)$ has a non-empty interior for some $\pi \in \convexcombm{m}$. Hence, from Sard's theorem, the regular values of the marginal map $t(\mathcal{F},\cdot)$ are dense in $T^{(m)}(\mathcal{F},\pi)$. Hence there is a regular point $(\sfmat,\pi) \in \stepfunctionspacelpi{m}{\pi}$ of  $t(\mathcal{F},\cdot)$. This implies that $\mathcal{F}$ has independent graphs. 

\end{proof}

\lemnonzerogradient*
\begin{proof}
From the partial derivative formula given on (\ref{eq:1partialderstep}), we have
\begin{eqnarray*}
&& \langle (\sfmat,\pi), \nabla t(F, (\sfmat,\pi) \rangle  \\
&& = \sum_k \sum_{i_1 \leq \cdots i_{k} } \pi_{i_1} \cdots \pi_{i_{d_k}} \sfmat_{i_1 \cdots i_{d_k}, k} \sum_{(j_1, \cdots,j_{d_k}) \in orb((i_1 \cdots i_{d_k})) }  \\ && t_{j_1 \cdots j_{d_k}} (\partial_k^{  (\bullet \cdots \bullet)} F, (\sfmat,\pi) ) \\ && = 
\sum_k   \sum_{(a_1 \cdots a_{d_k}) \in E_k(F)}  \sum_{i_1 \leq \cdots \leq i_{k} } \pi_{i_1} \cdots \pi_{i_{d_k}} \sfmat_{i_1 \cdots i_{d_k}, k} \\ && 
t_{j_1 \cdots j_{d_k}} (F^{\bullet(a_1 \cdots a_{d_k})} , (\sfmat,\pi) ) \\ && = 
\sum_k   \sum_{(a_1 \cdots a_{d_k}) \in E_k(F)}  \sum_{i_1  \cdots i_{k}=1 }^m 
t_{j_1 \cdots j_{d_k}} (F^{\bullet a_1, \cdots, a_{d_k} } , (\sfmat,\pi) ) \pi_{i_1} \cdots \pi_{i_{d_k}} \\ && = \sum_k   \sum_{(a_1 \cdots a_{d_k}) \in E_k(F)} t(F^, (\sfmat,\pi) ) = |E(F)| t(\sfmat,\pi))
\end{eqnarray*}

\end{proof}

\lemsmoothnessalmosteverywhere*
\begin{proof}
From Lemma \ref{lem:nonzerogradient}, we have 
\begin{equation*}
\langle (\sfmat,\pi), \nabla t( \sum_k a_{k} \mathcal{F}_{k}, (\sfmat,\pi) \rangle = \sum_k e_k a_{1k} t(\mathcal{F}_{k},(\sfmat,\pi))    
\end{equation*}
where $e_k = |E(F_k)|$.  To prove $\realfeasibleregionmllpi{m}{ [\mathcal{F} }{ [u]  }{  \pi  }$ is smooth almost everywhere, it is sufficient to prove that there is a  $(\sfmat, \pi) \in \stepfunctionspacel{m}  $ such that they satisfy the following conditions 
\begin{equation}
\label{eq:var002}
    \sum_k a_k t(\mathcal{F}_k,  (\sfmat,\pi )) = \sum_k  a_k t(\mathcal{F}_k,  (\sfmat,\pi )) = u.
\end{equation}
and
\begin{equation}
\label{eq:var001}
\langle (\sfmat,\pi), \nabla t( \sum_k a_{k} \mathcal{F}_{k}, (\sfmat,\pi) \rangle = \sum_k  e_k a_{k}  t(\mathcal{F}_{k},(\sfmat,\pi)) \neq 0
\end{equation}
Let $c(F)=\{F_1, \cdots, F_n\}$ be the constituents of the quantum graph $F$.  Note that if we consider each $ (\mathcal{F}_k,  (\sfmat,\pi))$ as a variable $z_k$ then the constraints $(\ref{eq:var001})$ and $(\ref{eq:var002})$ become a hyperplane  and hyperplane complement in $\mathbb{R}^{|c(F)|}$ i.e. $\sum_k a_k z_k =u$  and $\sum_k a_ke_k z_k \neq 0$. Since $u \neq 0$, the hyperplane is not contained in the hyperplane complement. Hence we can find step functions $(\sfmat,\pi) \in \realstepfunctionspacel{m}$ satisfying $(\ref{eq:var001})$ and $(\ref{eq:var002})$.
    
\end{proof}

\thmnumbercomponents*
\begin{proof}

Note that
\begin{equation*}
\realfeasibleregionml{m+1} = \bigcup_{\tiny \begin{array}{c}
    \pi \in \convexcombm{m} \\ \lambda \in [0,1]   
\end{array} } \realfeasibleregionmlpi{m+1}{\theta(\pi,\lambda,k)}.    
\end{equation*}
Hence it is sufficient to prove that if $\lambda \in (0,1)$, then the number of connected components of 
 $ \realfeasibleregionmlpi{m+1}{\theta(\pi,\lambda,k)}$  is not higher than the number of connected components of 
 $ \realfeasibleregionmlpi{m}{\theta(\pi,0,m)}$. By contradiction. There is  a new connected component, say 
\begin{equation*}
N_1 \subset \realfeasibleregionmlpi{m+1}{\theta(\pi,\lambda_1,m)}.    
\end{equation*}
The existence of $N_1$ is due to the new connected component $N \subset \realfeasibleregionml{m+1} \setminus \realfeasibleregionml{m}$ containing $N_1$, otherwise $N_1$ is not a new connected component. 
Let 
\begin{equation*}
    \lambda^* = \inf \{ \lambda  : N \cap  
    \realfeasibleregionmlpi{m+1}{ \theta(\pi,\lambda,m)) }
 \neq \emptyset\}.
\end{equation*}

If $\lambda^* = 0$ then there is $x \in N$ arbitrary close to $\realfeasibleregionml{m} \subset T^{-(m_0, m)}$. Hence there is a sequence $\{x_n \}_{n=1}^\infty \subset N \subset T^{-(m_0,m+1)}(\mathcal{F}) \setminus T^{-(m_0,m)}(\mathcal{F})$  such that $t(\mathcal{F},x_n)=u$ for all $n >0$ and  the distance to $T^{-(m_0,m)}(\mathcal{F})$ converges to zero i.e. $d(x_n, T^{-(m_0,m)}(\mathcal{F})) \to 0$  as $n \to \infty$. Hence we conclude that 
$u \in \partial T^{(m)}(\mathcal{F})$ which is a contradiction since $u \in T^{(m_0)}(\mathcal{F})^\circ \subset T^{(m)}(\mathcal{F})^\circ$. Thus $\lambda^* >0$.

First we assume $u \in \regstatsm{m}{\mathcal{F}}$. 
Let $(\sfmat, \theta(\pi,\lambda^*,m) ) \in N \cap \realfeasibleregionmlpi{m+1}{\theta(\pi,\lambda^*,m)}$. Since $u \in Reg(\mathcal{F})$,  $ \realfeasibleregionml{m}$ is a smooth manifold of $n(m,r,d)+m-1-|\mathcal{F}|$  dimensions. 

Let $U \subset N$ be an open neighborhood  of $(\sfmat, \theta(\pi,\lambda^*,m) ) $ in $\realfeasibleregionml{m+1}$. Moreover note that $\realfeasibleregionmlpi{m}{\pi} \cap \realfeasibleregionmlpi{m}{\pi'} = \emptyset$ if $\pi \neq \pi'$ and 
\begin{equation*}
\realfeasibleregionml{m}=\bigsqcup_{\pi \in \convexcombm{m}} \realfeasibleregionmlpi{m}{\pi}.    
\end{equation*}
Let $Exp^{-1}_x:U \to T_x \realfeasibleregionml{m+1}$ be the inverse of exponential map (see Eq. (  \ref{eq:exponentialmap})) defined on a neighborhood of $U$ of $x$.
From Theorem  \ref{thm:smoothnessalmosteverywhere}, $\realfeasibleregionmlpi{m+1}{\theta(\pi,\lambda^*,m)}$  is smooth almost everywhere, but with non-uniform dimension, each connected component could have a different dimension. 

If $x=(\sfmat, \theta(\pi,\lambda^*,m) ) $ is a regular point of  $\realfeasibleregionmlpi{m+1}{\theta(\pi,\lambda^*,m)}$ at the connected component $C \subset \realfeasibleregionmlpi{m+1}{\theta(\pi,\lambda^*,m)}$ which has $n_1$ dimensions  and $n_1 \geq n(m+1, r,d) -|\mathcal{F}|$ then 
\begin{equation*}
T_{x} C \subseteq T_{x}\realfeasibleregionml{m+1}.    
\end{equation*}
Recall that $\realfeasibleregionml{m+1} $ has  $n(m+1,r,d) +m-|\mathcal{F}|$ dimensions.  Let $O_x$ be the orthogonal complement  of $T_{x} C$ in $ T_{x}\realfeasibleregionml{m+1}$. 
Let $U_1 \subset O_x$ be an open neighborhood of $0$.  Note that $O_x$ has $n_2 = n(m+1,r,d) -|\mathcal{F}| -n_1$ dimensions and $n_2 \in [m]$. It follows that only $n_2$ vectors in $T_{x}\realfeasibleregionml{m+1}$ can vary the parameters of the partition vectors\footnote{The number of parameters of partition vectors in $\convexcombm{m}$ is $m-1$ since $\sum_i \pi_i =1 $.} $\pi' \in \convexcombm{m+1}$ of the step functions in $\realfeasibleregionml{m+1}$ through the exponential map $Exp_x:V  \to U $ where $V \subset T_x \realfeasibleregionml{m+1}$ is a neighborhood of $0$.  

If $n_2=0$ then $T_{x} C = T_{x}\realfeasibleregionml{m+1}$. It follows that $U = U \cap C $ and since $\realfeasibleregionml{m+1}$ and $C$ are analytical manifolds then $N=C$. If dimension of $C$ is higher than $n(m+1,r,d) -|\mathcal{F}|$ then there are at most $|\mathcal{F}|-1$ active constraints but on  $\realfeasibleregionml{m+1}$ there are $|\mathcal{F}|$ active constraints. This a contradiction.  Thus, $n_2$ is at least equal to $1$.

In the worst case when $n_2=1$, only one parameter in $\convexcombm{m+1}$ can be varied by the vectors in $O_x \subset T_{x}\realfeasibleregionml{m+1}$.

Recall that every $(\sfmat',\pi') \in \realfeasibleregionml{m+1} $ is invariant  under permutations of $m+1$ row/columns. Hence, we have the freedom to choose which coordinates of the partition vectors are varied by vectors in $O_x$. Let's assume the vectors in $O_x$  vary the $m$-coordinate of $\convexcombm{m+1}$. Hence these vectors in $O_x$  correspond to the variation of parameter $\lambda$ in $(\sfmat, \theta(\pi,\lambda,m) )$. Hence we conclude that $U$ can contain step functions with partition vectors of the form $(\sfmat',\theta(\pi,\lambda_3,m) ) \in U$ where  $\lambda_3 < \lambda^*$. 

If $x=(\sfmat, \theta(\pi,\lambda^*,m) )$ is a singular point of $\realfeasibleregionmlpi{m+1}{\theta(\pi,\lambda^*,m)}$ and since $\realfeasibleregionmlpi{m+1}{\theta(\pi,\lambda^*,m)}$ is smooth  almost everywhere then there is a regular point  $z \in \realfeasibleregionmlpi{m+1}{\theta(\pi,\lambda^*,m)}$ arbitrary close to $x$. Hence we apply the arguments for the case $x=(\sfmat, \theta(\pi,\lambda^*,m) )$ and conclude $Exp_x(U)$  contains step functions with partition vectors of the form $(\sfmat',\theta(\pi,\lambda_3,m) ) \in U$ where  $\lambda_3 < \lambda^*$. 

Both cases contradict that $\lambda^*$ is the infimum of  $\lambda$ such that
\begin{equation*}
    N \cap \realfeasibleregionmlpi{m+1}{\theta(\pi,\lambda,m)} \neq \emptyset.
\end{equation*}

Now we assume $u \in \cristats{m}{\mathcal{F}}$. 

From Remark \ref{rem:aggregatedfeasibleregion}, there are arbitrary close regular points $u'$ to $u$ such that the points of $\realfeasibleregionmq{m+1}{\mathcal{F}}{u'}$
 are arbitrary close to points of  $\realfeasibleregionml{m+1}$. For the feasible region $\realfeasibleregionmq{m+1}{\mathcal{F}}{u'}$
  there is no new connected component in $\realfeasibleregionmq{m+1}{\mathcal{F}}{u'} \setminus \realfeasibleregionmq{m}{\mathcal{F}}{u'}$ . Hence, we conclude that $\realfeasibleregionml{m+1}$ has no new connected component.

\end{proof}

\lemcriticalsplit*

\begin{proof}

In the first case. Without loss generalization, we assume that $k=m$. It is sufficient to prove that $\theta((\sfmat^*,\pi), \lambda,m)$ is a critical point when if $(\sfmat^*,\pi^*)$ is critical. 

Recall the formulas to compute the first partial derivative from Theorem \ref{eq:1partialderstep} we have,
\begin{eqnarray*}
\frac{\partial t(F)}{\partial \sfmat_{i_1, \cdots, i_{d_k},k} } &=& \pi_{i_1} \cdots \pi_{i_{d_k}} \sum_{(j_1, \cdots,j_{d_k}) \in orb((i_1 \cdots i_{d_k})) } t_{j_1 \cdots j_{d_k}} (\partial_k^{  (\bullet \cdots \bullet)} F, (\sfmat,\pi) )
\end{eqnarray*}
and by taking the partial derivatives for  (\ref{eq:generalratefunctionstep}), we have
\begin{equation*}
\frac{\partial f_s((\sfmat,\pi))}{\partial \sfmat_{i_1, \cdots i_{d_k},k}} = \pi_1 \cdots \pi_{d_k}  f'_0(\sfmat_{i_1 \cdots i_{d_k} }) 
\end{equation*}
Hence the constraint 
\begin{equation*}
  \frac{\partial f_s((\sfmat,\pi))}{\partial \sfmat_{i_1 \cdots i_{d_k},k}}= \sum_s \beta_s  \frac{\partial t(\mathcal{F}_s,(\sfmat, \pi))}{\partial \sfmat_{i_1 \cdots i_{d_k}},k}  
\end{equation*}
is reduced to 
\begin{equation*}
   f'_0(\sfmat_{i_1 \cdots i_{d_k},k })  = \sum_s \beta_s    \sum_{(j_1, \cdots,j_{d_k}) \in orb((i_1 \cdots i_{d_k})) } t_{j_1 \cdots j_{d_k}} (\partial_k^{ (\bullet \cdots \bullet)} \mathcal{F}_s, (\sfmat,\pi) )  
\end{equation*}
Note that the above constraint depends only on the step function $(\sfmat,\pi)$ and not on their parameterization, thus   
\begin{equation}
   \label{eq:constraintAij}
    \frac{\partial f_s((\sfmat,\pi))}{\partial \sfmat_{i_1 \cdots i_{d_k},k}}\big{|}_{(\sfmat,\pi)=\theta((\sfmat^*,\pi^*), \lambda,m)}  = \sum_s \beta_s  \frac{\partial t(\mathcal{F}_s,(\sfmat, \pi))}{\partial \sfmat_{i_1 \cdots i_{d_k},k}} \big{|}_{(\sfmat,\pi)=\theta((\sfmat^*,\pi^*), \lambda,m)}    
\end{equation}
for all $i_1, \cdots, i_{d_k} \in [m]$.
Now to compute the Lagrange multiplier condition for $\frac{\partial }{\partial \pi_i}$, we need to consider the constraint $\sum_k \pi_k=1$, thus
\begin{equation*}
    \frac{\partial }{\partial \pi_i}\big{|}_{\sum_k \pi_k=1} = \frac{\partial }{\partial \pi_i} - \frac{\partial }{\partial \pi_m}
\end{equation*}
Hence
\begin{equation*}
    \frac{\partial f_s}{\partial \pi_i} - \frac{\partial f_s}{\partial \pi_m} = \sum_k \beta_k ( \frac{\partial t(\mathcal{F}_k)}{\partial \pi_i} - \frac{\partial t(\mathcal{F}_k)}{\partial \pi_m} )
\end{equation*}
The partial derivatives $\frac{\partial}{\partial \pi_i}$ of $t(F,\cdot)$ and $f_s$ are
\begin{equation*}
\frac{\partial t(F,(\sfmat, \pi))}{\partial \pi_i} = t_i(\partial^\bullet F,(\sfmat, \pi))
\end{equation*}
and 
\begin{equation*}
\frac{\partial f_s((\sfmat,\pi))}{\partial \pi_i} = \sum_k  d_k \sum_{i_2 \cdots i_{d_k}=1}^m f_0(A_{i, i_2,  \cdots i_{d_k} } ) \pi_{i_2} \cdots \pi_{i_{d_k}}
\end{equation*}
Note that the partial derivatives $\frac{\partial}{\partial \pi_i}$ depend only on the step function $(\sfmat,\pi)$ and not on their parameterization, thus   
\begin{equation}
\label{eq:constraintpi}
    \frac{\partial f_s}{\partial \pi_i} - \frac{\partial f_s}{\partial \pi_m} \big{|}_{(\sfmat,\pi)=\theta((\sfmat^*,\pi^*), \lambda,m)}     = \sum_k \beta_k ( \frac{\partial t(\mathcal{F}_k)}{\partial \pi_i} - \frac{\partial t(\mathcal{F}_k)}{\partial \pi_m} ) \big{|}_{(\sfmat,\pi)=\theta((\sfmat^*,\pi^*), \lambda,m)}    
\end{equation}
From (\ref{eq:constraintAij}) and (\ref{eq:constraintpi}), we conclude that 
\begin{equation}
\label{eq:lagrangemultipliersplit}
    \nabla f_s((\sfmat,\pi)) \big{|}_{(\sfmat,\pi)=\theta((\sfmat^*,\pi^*), \lambda,m)}      = \sum_k \beta_k \nabla t(\mathcal{F}_k,(\sfmat,\pi)) \big{|}_{(\sfmat,\pi)=\theta((\sfmat^*,\pi^*), \lambda,m)} 
\end{equation}
The above constraint proves the first case. For the second case, the proof is straightforward, which is omitted. 
\end{proof}

\thmalmostsmoothnessCr*
\begin{proof}
First, we give sufficient conditions for the smoothness of $\criticalpoints{m}$.  Note that   $\criticalpoints{m}$ is defined by the constraints 
\begin{equation}
\label{eq:varA}
\frac{\partial f_s((\sfmat, \pi))}{\partial \sfmat_{i_1  \cdots i_{d_k},k}} - \sum_{k=1}^{|\mathcal{F}|} \beta_j \frac{\partial t(\mathcal{F}_j, (\sfmat,\pi))}{\partial \sfmat_{i_1 \cdots i_{d_k},k}}=0 \quad \mbox{ for } 1 \leq  i_1  \cdots \leq i_{d_k} \leq m
\end{equation}
and
\begin{equation}
\label{eq:varpi}
\frac{\partial f_s((\sfmat, \pi))}{\partial \pi_i} - \sum_{k=1}^{|\mathcal{F}|} \beta_k \frac{\partial t(\mathcal{F}_k, (\sfmat,\pi))}{\partial \pi_i} =0 \quad \mbox{ for } 1 \leq  i  \leq m-1
\end{equation}
in the ambient space $\mathbb{R}^{n(m,r,d)+m-1}+|\mathcal{F}|$.  Thus, if $\criticalpoints{m}$ is a manifold, it must have  $|\mathcal{F}|$ dimensions independent of  $m$. Hence, we conclude that the split of critical points does not bifurcate other critical points.  

A sufficient condition for the smoothness $\criticalpoints{m}$  comes from the Constant Rank Theorem. Hence we compute the  partial derivatives $\frac{\partial}{\partial \sfmat_{i_1 \cdots i_{d_k}}}$,  $\frac{\partial}{\partial \pi_i}$ and  $\frac{\partial}{\partial \beta_i}$to the constraints $(\ref{eq:varA})$  and $(\ref{eq:varpi})$ to obtain the Jacobian $J =J_{((\sfmat,\pi),\beta)}(\mathcal{F},f_s) $ of the constraints
\begin{equation*}
J  = \left[ \begin{array}{c|c}
     H_{((\sfmat,\pi),\beta)}L(f_s,\beta)  &  -J_{((\sfmat,\pi))}(\mathcal{F})^\top
\end{array} \right]
\end{equation*}
Note that $H_{((\sfmat,\pi),\beta)}L(f_s,\beta)=H^g_{(\sfmat,\pi)} f_s$ at $(\sfmat,\pi)$ and  $J_{(\sfmat,\pi)}(\mathcal{F},\pi)$ is  full-rank on $\realfeasibleregionml{m}$  since  $u \in Reg(\mathcal{F})$. Thus, we must prove that $J $ has a constant rank on $\criticalpoints{m}$. 

Since the constraints (\ref{eq:constraintAij}) and (\ref{eq:constraintpi}) are defined by analytical functions,  the rank of $J$ might be different in each connected component $C_k(m)$ of $\criticalpoints{m}$. In this case, in each $C_k(m)$, the matrix  $J$ might have redundant columns. These redundant columns must be columns from $H_{((\sfmat,\pi),\beta)}L(f_s,\beta) $ since $J$ is full rank. Hence the corresponding constraints from $(\ref{eq:varA})$  and $(\ref{eq:varpi})$ are redundant. Let $V_k(m)$ be the set of coordinates $\sfmat_{ij}$ and $\pi_i$ of $C_k(m) $ that correspond to redundant constraints. 

The redundant constraints appear when there are non-isolated critical points since the coordinates from $V_k(m)$ can have any value and do not change the condition that $x \in \criticalpoints{m}$ is a critical point. By dropping the redundant coordinates, we conclude the variation of these coordinates does not bifurcate the critical points in $C_k(m)$. Note that from (\ref{eq:lagrangemultipliersplit}) in proof Lemma \ref{lem:criticalsplit}, the split operation only adds redundant constraints. Thus $\criticalpoints{m}$ is smooth almost everywhere and each connected component  $C_k$ has a  dimension $|\mathcal{F}| + |V_k(m_0)|$. 


Hence, the split operation can not bifurcate critical points in  $C_k$ for almost all cases when $u \in T(\mathcal{F})$. Since any region defined by analytical constraints has at most a countable number of connected components, we conclude the split operation does not bifurcate into another critical point for almost all cases when $u \in T(\mathcal{F})$.

 
\end{proof}

\lemsplitminimum*

\begin{proof}
\quad
\begin{enumerate}
    \item Since $(\sfmat^*,\pi^*)$ is a local minimum on $ \realfeasibleregionmlpi{m}{\pi^*}$, hence from Theorem \ref{thm:necessarysecondordercondition}, we have 
\begin{equation*}
v^\top \left(  H^g_{(\sfmat^*,\pi^*)} f_s\right) v \geq 0    
\end{equation*}
for any $v \in C_{ (\sfmat^*,\pi^*)} \realfeasibleregionmlpi{m}{\pi^*}$. Hence  $v^\top \left( H^g_{(\sfmat^*,\pi^*)} f_\epsilon \right) v > 0$. Hence there is a $\lambda_1 > 0$ such that $v^\top \left( H^g_{\theta(\sfmat^*,\pi^*),\lambda_1,k)} f_\epsilon \right) v > 0$ for any  \\ $v \in C_{ \theta(\sfmat^*,\pi^*),\lambda_1,k)} \realfeasibleregionmlpi{m}{\pi^*}$.

\item Let $U(\pi^*) \subset \convexcombm{m}$ be an open neighborhood such that for any $\pi \in U(\pi^*)$  there is an isolated local minimum $W^*(\pi)$ of  $f_\epsilon$ on $\realfeasibleregionmlpi{m}{\pi}$.

\item Let $W(\pi,\lambda,k)=\theta(W^*(\pi), \lambda, k)$ and  $\pi(\lambda,k)=\theta(\pi, \lambda, k)$, hence $W(\pi^*,0,k)=(\sfmat^*,\pi^*)$. 
Since $v^\top \left( H^g_{(\sfmat^*,\pi^*)} f_\epsilon \right) v >0$ for any $v \in C_{ (\sfmat^*,\pi^*)} \realfeasibleregionmlpi{m}{\pi^*}$,  there is  $\lambda_1 >0$ such that  $v^\top \left(H^g_{\theta((\sfmat^*,\pi^*),\lambda_1,k)} f_\epsilon \right) v >0$ for any 
\begin{equation*}
v \in C_{W(\pi^*,\lambda_1,k)} \realfeasibleregionmlpi{m}{\pi^*(\lambda_1,k)}    
\end{equation*}
and  for all $k \in [m]$. Hence $W(\pi^*,\lambda_1,k)= \theta((\sfmat^*,\pi^*),\lambda_1,k)$.

\item We claim that $W(\pi^*,\lambda_1/2,k)= \theta((\sfmat^*,\pi^*),\lambda_1/2,k)$ is a local minimum  of  $f_\epsilon$ on $\realfeasibleregionml{m+1}$. Let $V$  be an open neighborhood of $W(\pi^*,\lambda_1/2,k)$. By contradiction, let $(\sfmat',\pi') \in \realfeasibleregionml{m+1}$ be a different point from $W(\pi^*,\lambda_1/2,k)$ in a neighborhood $V$ of $W(\pi^*,\lambda_1/2,k)$ such that  . 
\begin{equation}
    \label{eq:inequality001}
    f_\epsilon((\sfmat',\pi')) < f_\epsilon(W(\pi^*,\lambda_1/2,k)).
\end{equation}
Since $V$ can be arbitrary small, $\pi'$ can be arbitrary close to  $\theta(\pi,\lambda_1/2,k)$. Hence there is a $\pi'' \in U(\pi^*) \subset \convexcombm{m}$ such that $\pi' = \theta(\pi'',\lambda_2/2,k)$ and for some $\lambda_2 \in (0,\lambda_1)$.
    
By definition,  $ W^*(\pi'')$ is a local minimum of $f_\epsilon$ on $\realfeasibleregionmlpi{m}{\pi''}$. Hence, we have the chain of inequalities,
\begin{multline}
f_\epsilon(W(\pi^*,\lambda_1/2,k))= f_\epsilon((\sfmat^*,\pi^*)) \\ \leq f_\epsilon(W^*(\pi'')) = f_\epsilon(W(\pi'',\lambda_2/2,k))   \leq f_\epsilon((\sfmat',\pi')).
\end{multline}
Hence we have $f_\epsilon(W(\pi^*,\lambda_1/2,k)) \leq f_\epsilon((\sfmat',\pi'))$ which contradicts (\ref{eq:inequality001}). Hence $W(\pi^*,\lambda_1/2,k)=\theta((\sfmat^*,\pi^*),\lambda_1/2,k)$ is a local minimum of  $f_\epsilon$ on $\realfeasibleregionml{m+1}$ and thus we set $\lambda_0$ to $ \lambda_1/2$.  

\end{enumerate}

\end{proof}

\thmsplitminimum*
\begin{proof}
\quad

\begin{enumerate}


\item Let $f_\epsilon$ be a perturbed $f_s$ given in Definition \ref{def:fdperturbation}. From Lemma \ref{lem:splitminimum}, there is $\lambda_0 > 0$ such that $(\sfmat,\theta(\pi,\lambda_0,k))$ is a local  minimum of $f_\epsilon$ on $\realfeasibleregionml{m+1}$.

\item Now, we claim that $\theta((\sfmat,\pi),\lambda,k)$ is a local  minimum of $f_\epsilon$ on $\realfeasibleregionml{m+1}$ for all $\lambda \in (0,1)$. 
   
Since $u \in Reg^{(m)}(\mathcal{F})$, the tangent space $T_{(\sfmat,\theta(\pi^*,\lambda,k))} \realfeasibleregionml{m+1}$ is defined for all $\lambda \in (0,1)$. Let $U$ be a local neighborhood of $(\sfmat, \theta(\pi^*,\lambda_0, k ))$ in $ \realfeasibleregionml{m+1}$. 

Let $\phi:U \to \realfeasibleregionml{m+1}$ be a local diffeomorphism such that 
\begin{equation*}
\phi( (\sfmat, \theta(\pi^*,\lambda_0, k )) \,)=(\sfmat, \theta(\pi^*,\lambda, k )).    
\end{equation*} 
Hence $\phi^{-1}: \phi(U) \to U$ is a diffeomorphism. From Lemma \ref{lem:diffeomorphismpreservation}, local diffeomorphism preserves critical points and $H^g_{(\sfmat, \theta(\pi^*,\lambda, k ))} f_\epsilon  $ is semidefinite positive on $T_{(\sfmat, \theta(\pi^*,\lambda, k ))} \realfeasibleregionml{m+1}$ iff $H^g_{(\sfmat, \theta(\pi^*,\lambda_0, k ))} f_\epsilon \circ \phi^{-1} $ is semidefinite positive on $T_{(\sfmat, \theta(\pi^*,\lambda_0, k ))} \realfeasibleregionml{m+1}$.

Hence to prove that $\theta( (\sfmat^*,\pi^*),\lambda, k ))$ is local minimum of $f_\epsilon$ in $\realfeasibleregionml{m+1}$, it is sufficient to prove that $\theta( (\sfmat^*,\pi^*),\lambda_0, k ))$ is a local minimum of $f_\epsilon \circ \phi^{-1}$. 
 
 Note that for every open neighborhood $V \subset \phi(U)$, we have $\phi^{-1}(V) \subset U$ and since $\theta( (\sfmat^*,\pi^*),\lambda_0, k ))$ is a local minimum of $f_\epsilon$ in $U$, thus  $\theta( (\sfmat^*,\pi^*),\lambda_0, k ))$ is a local minimum of $f_\epsilon \circ \phi^{-1}$ in $V \subset \realfeasibleregionml{m+1}$ for any $\lambda \in (0,1)$.



\begin{equation*}
\end{equation*} 



\end{enumerate}
\end{proof}

\lemcontinuity*
\begin{proof}

Let $W \in  \realgraphonspacel$ and $V \in U$ where $Y$ is an open neighborhood of $W$ sufficiently small. To prove both cases, it is sufficient to prove that there is  a constant $C >0$ such that 
\begin{equation*}
|f_s(W) - f_s(V)| \leq  C \|W-V \|_1 \mbox{ and }  |I(W) - I(V)| \leq  C \|W-V \|_1
\end{equation*}
Thus we have  
\begin{eqnarray*}
&& |f_s(W) - f_s(V)|= \left| \sum_{k=1}^r \int_{[0,1]^{d_k}} \left(f_0(W_k(x_1, \cdots x_{d_k} ))-f_0(V_k(x_1, \cdots x_{d_k}) \right) dx_1, \cdots dx_{d_k}\right| \\
&&\leq \sum_{k=1}^r \int_{[0,1]^{d_k}} |f_0(W_k(x_1, \cdots x_{d_k}))-f_0(V_k(x_1, \cdots x_{d_k}))| dx_1, \cdots dx_{d_k}  \\
&&\leq  C \sum_{k=1}^r \int_{[0,1]^{d_k}} |W_k(x_1, \cdots x_{d_k})-V_k(x_1, \cdots x_{d_k})| dx_1, \cdots dx_{d_k} =   C \|W-V \|_1.
\end{eqnarray*}
In the second case.  Note that $|I'_0(W)|$ and $|I'_0(V)|$ can be bounded by $|I'_0(\Delta(W,V))|$ where 
\begin{eqnarray*}
\Delta(W,V) = \min_{k \in [r]} \inf_{(x_1, \cdots x_{d_k}) \in {[0,1]}^{d_k} }  \{|W_k(x_1, \cdots, x_{d_k})|, \\|1-W_k(x_1, \cdots, x_{d_k})|,|V_k(x_1, \cdots, x_{d_k})|,|1-V°k(x_1, \cdots, x_{d_k})|\}    
\end{eqnarray*}
and $I'_0(x) = \log\left( \frac{x}{1-x} \right) $. Thus 
\begin{eqnarray*}
&&|I(W) - I(V)|= \left| \sum_{k=1}^r \int_{[0,1]^{d_k}} \left(I_0(W_k(x_1, \cdots x_{d_k}))-I_0(V_k(x_1, \cdots x_{d_k})) \right) dx_1, \cdots dx_{d_k} \right| \\
&&\leq \sum_{k=1}^r\int_{[0,1]^{d_k}} |I_0(W_k(x_1, \cdots x_{d_k}))-I_0(V_k(x_1, \cdots x_{d_k}))| dx_1, \cdots dx_{d_k}  \\
&&\leq \sum_{k=1}^r |I'_0(\Delta(W,V))| \int_{[0,1]^{d_k}} |W_k(x_1, \cdots x_{d_k})-V_k(x_1, \cdots x_{d_k})| dx_1, \cdots dx_{d_k} \\
&=&   |I'_0(\Delta(W,V))| \cdot \|W-V \|_1.
\end{eqnarray*}

\end{proof}

\thmPOD*

\begin{proof}
\quad
\begin{enumerate}
    \item 
Let $W_{m}$ be a global minimum of $f_s$ on $\realfeasibleregionml{m}$. 
By iterating  Theorem  \ref{thm:localminimumsplit}, 
starting at $m=m_0(\mathcal{F})$, we obtain a sequence of local minima $(W_m)$ that of the same step function $W_{m_0}$.  Hence $W_{m_0}$ is a global minimum of $f_s$ in $\cup_{m\geq m_0} \realfeasibleregionml{m}$.

\item We claim that $W_{m_0}$ is a global minimum of $\min_{W \in \realfeasibleregionl} f_s(W)$. By contradiction let assume that $W_1$ is a global minimum such that $c_1= f_s(W_1) < f_s(W_{m_0})=a$. From Lemma \ref{lem:continuity}, $f_s$ is continuous, hence the level set $B(\epsilon)=f^{-1}([c,\epsilon)) \cap \realfeasibleregionu$ is open   in $L^1( \prod_k [0,1]^{d_k})$ topology for any $\epsilon >0$. By the density of the step functions  in $L^1(\prod_k  [0,1]^{d_k})$ topology, there is  a step function $W_2  \in B(\epsilon)$ such that $f_s(W_2) < a = f_s(W_{m_0})$  for a sufficient small $\epsilon >0$ which is not possible since $W_0$ is a global minimum of all step functions. Thus $W_{m_0}$ is a global minimum of $f_s$ on $\realfeasibleregionl$ when we consider the neighborhoods of $W_{m_0}$ from  $L^1( \prod_k [0,1]^{d_k})$ topology. The same conclusion holds for cut-norm topology $\| \cdot \|_\Box$ because $L^1(\prod_k [0,1]^{d_k})$ topology is stronger than cut-norm topology, hence every open neighborhood of $W_{m_0}$ includes open neighborhood in cut-norm topology. 


\item Finally, it is sufficient to prove that it is not possible that $\optimalsolutionf{f_s}$ is a non-step function. Let $C^* = f_s(\optimalsolutionf{I})$. 

By contradiction. Let $V$ be a non-step function solution of $\optimalsolutionf{f}$. Hence $f_s(V)=C^*$. Let $d(x,y)=\ \|x-y \|_1$ be the metric on $L^1( \prod_k [0,1]^{d_k})$ norm.

Since $V$ be a non-step function, there is an $\epsilon > 0$ such that $d(V,x)$ for all $x \in \realstepfunctionspacel{m_0(\mathcal{F},u)}$. Hence, $V$ has an open neighborhood $U \subset \realgraphonspacel$ that contains step functions whose size is larger than  $m_0(\mathcal{F},u)$. Note that for every step function $x \in U$, we have $f_s(x) > C^*$ since $C^*$ is the global value.  Since the step functions are dense in $U$ in $L^1(\prod_k  [0,1]^{d_k})$ topology, let $(V_n)$ be a sequence of step functions that converge to $V$. Hence $\lim_{n \to \infty} f_s(V_n)$ is strictly higher than $C^*$ and $f_s(V)=C^*$. This implies that $f_s$ is not continuous. From Lemma  \ref{lem:continuity}, this is a contradiction.

\end{enumerate}

\end{proof}

\lemrandomsolutions*

\begin{proof}

\begin{figure}
    \centering
    \includegraphics[scale=0.5]{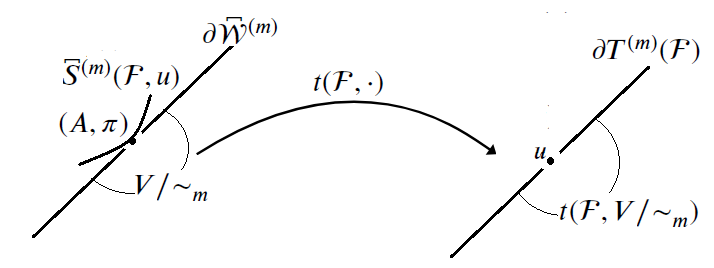}
    \caption{Since  $(\sfmat,\pi) \in \partial \stepfunctionspacel{m} \cap \feasibleregionml{m} $, $t(\mathcal{F},(\sfmat,\pi)) \in \partial T^{(m)}(\mathcal{F})$}
    \label{fig:randomsolutions}
\end{figure}

First we prove that if $(\sfmat,\pi) \in \partial \stepfunctionspacel{m} \cap \feasibleregionml{m} $ and $u \in T^{(m)}(\mathcal{F}) \cap \Omega(\mathcal{F})$ then there is  an open neighborhood $U$ of $(\sfmat,\pi)$  in $\realfeasibleregionml{m} $ such that $U \cap \randfeasibleregionml{m} \neq \emptyset$.

By contradiction. We assume that $U \cap \randfeasibleregionml{m} = \emptyset$. 
It follows that there is neighborhood  $V$ of $(\sfmat,\pi)$ with a non-empty interior in $\stepfunctionspacel{m}$ such that $V \cap \randfeasibleregionml{m} \cap = \emptyset$ as shown in Figure \ref{fig:randomsolutions} for the quotient space. From Lemma \ref{lem:openmarginal}, the image $t(\mathcal{F}, \cdot):\stepfunctionspacel{m} \to \mathbb{R}^{|\mathcal{F}|}$ on any open set has a non-empty interior. Now $t(\mathcal{F},V) \cap \partial T^{(m)}(\mathcal{F}) \neq \emptyset$ since $(\sfmat,\pi) \in \partial \stepfunctionspacel{m}$. Since $V$ can be arbitrary small, we conclude that $t(\mathcal{F},(\sfmat,\pi)) \in \partial T^{(m)}(\mathcal{F})$. This is a contradiction since $u \notin \partial T^{(m)}(\mathcal{F})$. Thus there is  an open neighborhood $U$ of $(\sfmat,\pi)$  in $\realfeasibleregionml{m} $ such that $U \cap \randfeasibleregionml{m} \neq \emptyset$. 

Without loss generalization we consider that $(\sfmat,\pi) \in \partial \stepfunctionspacel{m}$ and $(\sfmat,\pi)$ has only one zero entry i.e. $\sfmat_{1, \cdots 1,1} = 0$. We have proven that there exists a neighborhood $U$ in $\feasibleregionml{m}$ such that $\randfeasibleregionml{m} \cap U \neq \emptyset$. Take one step function $(C,\pi) \in \randfeasibleregionml{m} \cap U \neq \emptyset$. Let $\gamma$ be a rectifiable path connecting $(\sfmat,\pi)$  i.e. $\gamma(0)=(\sfmat,\pi)$ and $\gamma(1)=(C,\pi)$. Thus, it is sufficient to show that the line integral

\begin{equation}
\label{ineqEntropy}
I((\sfmat,\pi)) - I(\gamma(\epsilon))=\int_0^\epsilon dI(\gamma(s);\gamma'(s)) ds 
\end{equation}
is negative for some $\epsilon > 0$,
where $dI(W;V)$ is the the G\^ateaux derivative of $I(\cdot)$ at $W$ in the direction $V=\gamma'$, i.e.
\begin{equation*}
    dI(W;V) = \lim_{\lambda \to 0} \frac{I(W+\lambda V)-I(W)}{\lambda} 
\end{equation*}
hence we obtain 
\begin{equation*}
    dI(W;V) = \sum_{k=1}^r \int_{[0,1]^{d_k}} V_k(x_1, \cdots, x_{d_k}) I'_0(W_k(x_1, \cdots,x_{d_x} )  dx_1, \cdots,dx_{d_x} 
\end{equation*}
where $I'_0(u) = \log \left(\frac{u}{1-u} \right)$. We observe that 
\begin{itemize}
\item If $u \in (0,1)$ is fixed then $I'_0(u)$ is bounded  
\item  $\lim_{u \to 0} I'_0(u)=-\infty$
\item  $\lim_{u \to 1} I'_0(u)=+\infty$
\end{itemize}
Hence the definite integral (\ref{ineqEntropy}) is,
\begin{eqnarray*}
I((\sfmat,\pi)) &-& I(\gamma(\epsilon)) \\
&=&\int_0^\epsilon \left\lbrace \sum_{\tiny \begin{array}{c}
    i_1 \cdots i_{d_{1}}=1  \\
     i_1 \neq 1, \cdots i_{d_1} \neq 1 
\end{array}  }^m I'_0(\gamma(s)_{i_1 \cdots i_{d_1}}) \gamma'(s)_{i_1 \cdots i_{d_1} } \pi_{i_1}  \cdots \pi_{i_{d_1 }}  \right. \\ &&  \left. + I'_0(\gamma(s)_{1 \cdots 1,1}) \gamma'(s)_{1, \cdots 1,1} \pi_{1}^{d_1 }  \right\rbrace ds + K(\epsilon)
\end{eqnarray*}
where 
\begin{equation*}
    K(\epsilon ) = \int_{0}^\epsilon \sum_{k=2}^r \sum_{
    i_1 \cdots i_{d_k}=1 }^m    I'_0(\gamma(s)_{i_1 \cdots i_{d_1}})   \gamma'(s)_{i_1 \cdots i_{d_1} } \pi_{i_1}  \cdots \pi_{i_{d_k }}    ds 
\end{equation*}
and $\lim_{\epsilon \to 0^+} |K(\epsilon)| < \infty$.
Using the Mean Value Theorem, we obtain:
\begin{eqnarray*}
I((\sfmat,\pi)) &-& I(\gamma(\epsilon)) \\
&=& \epsilon \left\lbrace \sum_{\tiny \begin{array}{c}
    i_1 \cdots i_{d_{1}}=1  \\
     i_1 \neq 1, \cdots i_{d_1} \neq 1 
\end{array}  }^m I'_0(\gamma(\alpha)_{i_1 \cdots i_{d_1}}) \gamma'(\alpha)_{i_1 \cdots i_{d_1} } \pi_{i_1}  \cdots \pi_{i_{d_1 }}  \right. \\ &&  \left. + I'_0(\gamma(\alpha)_{1 \cdots 1,1}) \gamma'(\alpha)_{1, \cdots 1,1} \pi_{1}^{d_1 }  \right\rbrace ds + K(\epsilon)
\end{eqnarray*}
where $\alpha \in (0, \epsilon)$.
Note that the sum over all $1 \leq i_1, \cdots, i_{d_1} \leq m$ except $(1,\cdots,1)$ is bounded and the second adding is negatively unbounded when $\epsilon \to 0^+$ and then $\alpha \to 0^+$. Hence  there is a $\alpha >0 $ such that
$I((\sfmat,\pi)) - I(\gamma(\alpha)) <0$. 

\end{proof}

\end{document}